\documentclass[12pt]{amsart}
\usepackage[utf8]{inputenc}
\usepackage[english]{babel}
\usepackage[T1]{fontenc}
\usepackage{amsmath, amssymb}

\usepackage[a4paper,top=3cm,bottom=2cm,left=3cm,right=3cm,marginparwidth=1.75cm]{geometry}

\usepackage{amsthm}
\usepackage{amsfonts}
\usepackage{graphicx}
\usepackage[colorlinks=true, allcolors=blue]{hyperref}
\usepackage{url}
\newtheorem{theorem}{Theorem}
[section]

\newtheorem{corollary}[theorem]{Corollary}
\newtheorem{definition}[theorem]{Definition}
\newtheorem{lemma}[theorem]{Lemma}
\newtheorem{proposition}[theorem]{Proposition}

\theoremstyle{remark}
\newtheorem{remark}[theorem]{Remark}
\newtheorem{example}[theorem]{Example}

\def\CC{\mathbb{C}}

\def\NN{\mathbb{N}}
\def\QQ{\mathbb{Q}}
\def\RR{\mathbb{R}}
\def\ZZ{\mathbb{Z}}

\def\Qalg{\overline{\QQ}}

\def\calC{\mathcal{C}}
\def\calD{\mathcal{D}}
\def\calM{\mathcal{M}}
\def\calo{\mathcal{O}}
\def\calR{\mathcal{R}}
\def\calY{\mathcal{Y}}

\def\Cl{\mathrm{Cl}}
\def\Gal{\mathrm{Gal}}
\DeclareMathOperator{\CFF}{CFF} 
\DeclareMathOperator{\CFP}{CFP}

\title[Continued fractions for number fields]{On the finiteness of $\mathfrak{P}$-adic continued fractions for number fields}
\author{Laura Capuano}
\address{Dipartimento di Matematica e Fisica, Università degli Studi di Roma Tre}
\email{laura.capuano@uniroma3.it}
\author{Nadir Murru}
\address{Dipartimento di Matematica, Università di Trento}
\email{nadir.murru@unitn.it}
\author{Lea Terracini}
\address{Dipartimento di Informatica, Università di Torino}
\email{lea.terracini@unito.it}
\thanks{The three authors are members of the INdAM group GNSAGA. The first
author is funded by DISMA, Politecnico di Torino, Dipartimento di Eccellenza
MIUR 2018-2022.}
\keywords{$p$-adic continued fractions, finiteness, Weil height, norm Euclidean fields, norm Euclidean class \\
fractions continues $p$-adiques, finitude, hauteur de Weil, corps euclidiens pour la norme, classes euclidiennes}
\subjclass{11J70, 11D88, 11Y65, 11J72}
\begin{document}

\maketitle 

\begin{abstract}
For a prime ideal $\mathfrak{P}$ of the ring of integers of a number field $K$,
we give a general definition of $\mathfrak{P}$-adic continued fraction,  which also includes classical definitions of continued fractions in the field of $p$--adic numbers. We give some necessary and sufficient conditions on $K$ ensuring that for all but finitely many $\mathfrak{P}$, every $\alpha\in K$ admits a finite $\mathfrak{P}$-adic continued fraction expansion,  addressing a similar problem posed by Rosen in the archimedean setting. 
\end{abstract}

\section{Introduction} 
The classical continued fraction algorithm provides an integer sequence $[a_0, a_1, \ldots]$ that represents a real number $\alpha_0$ by means of the following recursive algorithm:
\begin{equation*} 
\begin{cases}
a_n = \lfloor \alpha_n \rfloor \cr
\alpha_{n+1} = \cfrac{1}{\alpha_n - a_n} \quad \mbox{if } \alpha_n-a_n\neq 0,
\end{cases}
\end{equation*}
for all $n \geq 0$, where $\lfloor\cdot\rfloor$ denotes the integral part of a real number. The $a_n$'s and $\alpha_n$'s are called \textit{partial} and \textit{complete} quotients, respectively. It is easy to see that, for classical continued fractions, the procedure eventually stops if and only if we start with a rational number, and, in the case of irrationals, they provide the best rational approximations of the number; this is one of the reasons why the study of continued fractions is very important in diophantine approximation and transcendence theory.

Motivated by this property, Rosen \cite{Rosen77} posed the problem of finding more general definitions of continued fraction expansions characterizing all the elements of an algebraic number field $K$ by means of finite expansions and providing approximations of elements not in the field by means of elements in $K$ (as well as classical continued fractions provide rational approximations of irrational numbers). In \cite{Rosen77}, Rosen gave an example of such continued fractions in the special case of $\mathbb Q(\sqrt{5})$, using expansions of the form
\[ a_0 + \cfrac{b_1}{a_1 \varphi + \cfrac{b_2}{a_2 \varphi + \cfrac{b_3}{\ddots}}}, \]
where $\varphi$ is the Golden ratio, $b_n = \pm 1$ and the $a_n\in \ZZ$ satisfies the property that $a_n\varphi$ is the integer multiple of $\varphi$ nearest to the respective complete quotients. This is a special case of the so called \textit{Rosen continued fractions}, introduced by the same author in \cite{Rosen54}, where $\varphi$ is replaced by irrational numbers of the form $2\cos\frac{\pi}{q}$ with $q\geq 3$  an odd positive number, with the aim of studying Hecke groups.

The characterization of the real numbers having a finite Rosen continued fraction is still an open problem, see, e.g., \cite{Arnoux09, Bug13, Hanson08} for further details. In \cite{Bernat06}, Bernat defined another continued fraction expansion in $\QQ(\sqrt{5})$, slightly different from the Rosen one, proving that also these continued fractions represent $\QQ(\sqrt{5})$ uniquely. Very recently in \cite{MVV2020} the authors generalized Bernat construction defining the so called $\beta$-continued fraction with the aim of studying when the elements of $\mathbb Q(\beta)$ have a finite representation, where $\beta$ is any quadratic Pisot number. More specifically, the authors proved that, if $\beta$ is either a quadratic Perron number or the square root of a positive integer, then every element of $\mathbb Q(\beta)$ has a finite or eventually periodic $\beta$-continued fraction expansion. Moreover, assuming a conjecture by Mercat \cite{Mercat13}, there exist only four quadratic Perron numbers $\beta$ such that the elements of $\mathbb Q(\beta)$ have finite $\beta$-continued fraction expansion.

\medskip
The problem of Rosen can be naturally translated into the context of $p$--adic numbers $\mathbb Q_p$; indeed, starting from Mahler \cite{Mahler1940}, continued fractions have been introduced and studied in $\mathbb Q_p$ by several authors. In this context, however, there is no natural definition of a $p$--adic continued fraction, since there is no canonical definition for a $p$-adic floor function. The two main definitions of a $p$--adic continued fraction algorithm are due to Browkin \cite{Browkin1978} and Ruban \cite{Ruban1970}; they are both based on the definition of a $p$--adic floor function
\[ 
s(\alpha) = \sum_{n = k}^0 x_n p^n \in \mathbb Q, \quad \mbox{where } \ \alpha = \sum_{n=k}^\infty x_n p^n \in \mathbb Q_p, 
\]
where the $x_n$'s are the representatives modulo $p$ in the interval $(-p/2, p/2)$ for Browkin definition and in the interval $[0, p-1]$ for Ruban definition. These continued fractions have been widely studied by several authors in terms of quality of the approximation, finiteness and periodicity, see, e.g., \cite{Bedocchi1988, Bedocchi1990, Browkin2000, CapuanoVenezianoZannier2019, CapuanoMurruTerracini2020, Laohakosol1985, Ooto02017, Tilborghs1990, DeWeger1988}. In this setting however, many differences with the classical case arise: for example, none of these definitions provide good approximations as in the real case, and no analogue of Lagrange's theorem holds for both Browkin and Ruban continued fractions, hence the problem of finding a standard definition for a $p$-adic continued fraction remains still open.
However, it has been proved that rational numbers have always finite Browkin continued fraction expansion \cite{Browkin2000} and finite or eventually periodic Ruban continued fraction expansion \cite{Laohakosol1985}. 
\medskip

In this paper, we consider the $p$-adic analogue of Rosen question. Given a number field $K$ and a prime ideal $\mathfrak{P}$ in  its ring of integers $\calo_K$, we give a very general definition of $\mathfrak{P}$-adic continued fractions; with this definition, the partial quotients are the values of a $\mathfrak{P}$-adic floor function $s$ which is a locally constant function from the $\mathfrak{P}$-adic completion of $K$ to the ring of $\{\mathfrak{P}\}$-integers of $K$. We will call the data $\tau=(K,\mathfrak{P},s)$ a \emph{type} and we will introduce the notion of continued fractions of type $\tau$. With this definition, Browkin and Ruban continued fractions arise as particular $p$-adic types for $\QQ$. If every element of $K$ has a finite (resp. periodic) $\tau$-expansion, then we shall say that the type $\tau$ enjoys the \emph{Continued Fraction Finiteness} ($\CFF$) (resp. \emph{Continued Fraction Periodicity} ($\CFP$)) property. Moreover, we shall say that the field $K$ enjoys the $\mathfrak{P}$-adic  $\CFF$ (resp. $\CFP$) property if there exists a $\CFF$ (resp. $\CFP$) type $\tau=(K,\mathfrak{P},s)$. It is well known that $\QQ$ satisfies the $p$-adic $\CFF$ property for every odd prime $p$ because of the finiteness of Browkin continued fraction expansions of rational numbers (see \cite{Browkin1978}). 
\medskip

In the first part of the paper, we prove a sufficient condition for a type to have the $\CFF$ (resp. $\CFP$) property using general properties of the multiplicative Weil height of algebraic numbers and of the norms of matrices. This result allows us to study the $\mathfrak{P}$-adic $\CFF$ property when $K$ is a norm Euclidean field; in particular, we prove that a norm Euclidean field with Euclidean minimum $< 1$ satisfies the $\mathfrak{P}$-adic $\CFF$ property for all but finitely many prime ideals $\mathfrak{P}$. Furthermore, for certain Euclidean quadratic fields $K$ we provide some more effective constructions by exploiting the form of unitary neighborhoods covering a fundamental domain of $\calo_K$ as done in \cite{Eggleton1992}.  

In the last part of the paper, we study the $\CFF$ property of $\mathfrak{P}$--adic continued fractions in relation with the structure of the ideal class group for number fields which are not necessarily norm Euclidean. First, we show that, if the number field $K$ satisfies the $\mathfrak{P}$-adic $\CFF$ property for all but finitely many prime ideals $\mathfrak P$, then $\calo_K$ is a PID, giving examples of number fields for which the $\CFF$ property fails to hold. Moreover, under milder hypotheses, we show that it is possible to ensure the $\CFF$ property for continued fractions associated to (almost all) primes belonging to a norm Euclidean class, in the sense of \cite{LenstraJr1979}. Finally, for a general number field, we show that the obstruction to the $\CFF$ property depends on the existence of infinitely many partial quotients with $\mathfrak{P}$-adic valuation equal to $-1$.
\medskip

We conclude this introduction by pointing out some open problems and directions for future work. First, effectiveness: our main results assert the existence of $\CFF$ types for a given number field, but in general it is not easy to define them explicitly. The construction of types satisfying the $\CFF$ property and the analysis of their properties, such as the study of the arithmetic of partial quotients, or of the dependence between the length of a finite continued fraction and the height of the algebraic number that it represents are interesting topics that we left outside of the scope of the present work. 
Moreover, it would be nice to obtain a full characterization of number fields $K$ satisfying the $\mathfrak{P}$-adic $\CFF$ for a given ideal $\mathfrak{P}$. We show that a necessary condition for $\CFF$ is that the ideal class group $K$ is generated by the class of $\mathfrak{P}$. We do not know if this condition is also sufficient, but we do not have arguments against this possibility.

Last, it would be nice to investigate periodicity: although we state a sufficient condition for a type to enjoy the $\CFP$ property, the present paper focuses specifically on finiteness. Nevertheless, periodicity is also a very interesting question, and an algebraic characterization of the elements represented by a periodic expansion (relatively to a given type) would be a challenging objective. 

\section{Notations and prerequisites}
For every rational prime $p$, let $|\cdot|_p$ be the $p$-adic absolute value, normalized in such a way that $|p|_p=\frac 1 p$.  The archimedean absolute values on $\RR$ or $\CC$ will be denoted by $|\cdot |$ or by $|\cdot |_\infty$ respectively. We will denote by $\overline{K}$ an algebraic closure of any field $K$.\\

Let $K$ be a number field of degree $d$ over $\QQ$, and let $\mathcal{O}_K$ be its ring of integers. We fix a prime ideal $\mathfrak{P}$ of $\mathcal{O}_K$ lying over an odd prime $p$.
Let $\calM_K$ be a set of representatives for the places of $K$. For every rational prime $q$ and every  $v\in\mathcal{M}_K$ above $q$ let $K_{v}\subseteq \overline{\mathbb{Q}}_q$ be the completion of $K$ w.r.t. the $v$-adic valuation and let $\calo_v$ be its valuation ring; we put $d_v=[K_v:\QQ_q]$. Let $|\cdot |_v=|N_{K_v/\QQ_q}(\cdot)|_q^{\frac 1{d_v}}$ be the unique extension  of $|\cdot |_q$ to $K_v$.  Then  the \emph{product formula} $$\prod_{v\in\mathcal{M}_K} |x|_v^{d_v}=1$$ holds for all $x\in K^\times$ (\cite[Prop. 1.4.4]{BombieriGubler2006}). We recall the definiton of multiplicative Weil height that will be useful in the paper.

\begin{definition}
  For $\alpha\in K $,
the (multiplicative) \emph{Weil height} is defined as
$$H(x) = \prod_{v\in\calM_K}
\sup(1, |x|_v)^{\frac {d_v} {d}}.$$
\end{definition}

Notice that all but finitely many factors of the infinite product are
equal to $1$, hence $H(x)$ is well defined. Moreover,  thanks to the choice of the normalization, the definition does not depend on the number field $K$, hence it extends to a
function $H : \Qalg \to [1, +\infty)$.
The function $H$ satisfies the following important properties (see \cite{BombieriGubler2006}):

\begin{proposition}\label{prop:propheight}
For every non-zero $x, y \in \Qalg$, we have:
\begin{itemize}
\item[a)] $H(x + y) \leq 2H(x)H(y)$;
\item[b)] $ H(xy) \leq  H(x)H(y)$;
\item[c)] $H(x^n) = H(x)^{|n|}$
for all $n \in\ZZ$;
\item[d)] $H(\sigma(x)) = H(x)$ for all $\sigma\in \Gal(\Qalg/\QQ)$;
\item[e)] {\bf Northcott’s theorem:} There are only finitely many algebraic numbers of bounded degree and bounded height.
\item[f)] {\bf Kronecker’s theorem:} $H(x) = 1$ if and only if x is a root of unity.
\end{itemize}
\end{proposition}

\section{$\mathfrak{P}$-adic continued fractions}\label{sect:genfacts}

In this section we show, given a number field $K$ and a prime ideal $\mathfrak{P}$ of $\calo_K$, how to define a general $\mathfrak{P}$-continued fraction. Our general definition will generalise the classical definitions of $p$-adic continued fractions given by Browkin and Ruban.

\subsection{$\mathfrak{P}$-adic floor functions and types} Let $\mathfrak{P}$ be a prime ideal of $\calo_K$, lying over an odd rational prime $p$ and let $v_0\in\calM_K$ be the place corresponding to $\mathfrak{P}$.
\begin{definition}
A \emph{$\mathfrak{P}$-adic floor function} for $K$ is a function $s:K_{v_0}\to K$ such that
\begin{itemize} 
\item[a)] $|\alpha-s(\alpha)|_{v_0}<1$ for every $\alpha\in K_{v_0}$;
\item[b)] $|s(\alpha)|_{v}\leq 1$ for every non archimedean $v\in \calM_K\setminus\{v_0\}$;
\item[c)] $s(0)=0$;
\item[d)] $s(\alpha)=s(\beta)$ if $|\alpha-\beta|_{v_0}<1$.
\end{itemize}
\end{definition}
 
We recall the strong approximation theorem in number fields \cite[Theorem 4.1]{Cassels}.
\begin{theorem} \label{teo:strong_approx}
Let $K$ be a global field and $\calM_K=S\cup T\cup \{w\}$ be a partition of the places of $K$ with $S$ finite. For every $v\in S$, let $a_v$ be an element in $K_v$ and let $\epsilon_v\in\RR_{>0}$. Then, there exists an $x$ in $K$ such that 
\begin{align*}
    |x-a_v|_v<\epsilon_v &\quad\hbox{for every } v\in S,\\
    |x|_v\leq 1 &\quad\hbox{for every } v\in T.
    \end{align*}
    \end{theorem}
    
By strong approximation (or some other arguments), $\mathfrak{P}$-adic floor functions always exist, and there are infinitely many. 

We define
\[\calo_{K,\{v_0\}}=\{\alpha\in K\ |\ |\alpha|_v\leq 1\hbox{ for every non archimedean } v\not=v_0 \hbox{ in } \calM_K\}.
\]
Then, we can regard a $\mathfrak{P}$-adic floor function as a map $s: K_{v_0}/\mathfrak{P}\calo_{v_0}\to \calo_{K,\{v_0\}}$ such that $s(\mathfrak{P}\calo_{v_0})=0$ and which is a section of the projection map $K_{v_0}\to K_{v_0}/\mathfrak{P}\calo_{v_0}$.
Therefore, the choice of a $\mathfrak{P}$-adic floor function amounts to choose  a set $\calY$ of representatives of the cosets of $\mathfrak{P}\calo_{v_0}$ in $K_{v_0}$ containing $0$ and contained in
$\calo_{K,\{v_0\}}$.\\
We shall call the data $\tau=(K,\mathfrak{P},s)$ (or $(K,\mathfrak{P},\calY)$) a \emph{($\mathfrak{P}$-adic) type}. 
\subsection{Types arising from generators of $\mathfrak{P}$} \label{sec:special_type}
In the case $\mathfrak{P}$ is principal, there is a more natural way of defining a floor function associated to $\mathfrak{P}$. Indeed, let $\pi\in\calo_K$ be  generator of $\mathfrak{P}$ and let $\mathcal{R}$ be a  complete set of representatives of $\mathcal{O}_K/\mathfrak{P}$ containing $0$. Then, every $\alpha\in K_{v_0}$ can be expressed uniquely as a Laurent series  $\alpha=\sum_{j=-n}^\infty c_j\pi^j$, where $c_j\in\mathcal{R}$ for every $j$. It is possible to define a  $\mathfrak{P}$-adic floor function by
$$s(\alpha)=\sum_{j=-n}^0c_j\pi^j\in K. $$
We shall denote the types $\tau=(K,\mathfrak{P},s)$ obtained in this way by $\tau=(K,\pi,\calR)$, and we will usually call them \emph{special types}.

\begin{example}[Browkin and Ruban types over $\mathbb{Q}$.]
When $K=\QQ$ and $\pi=p$ odd prime, two main special types have been studied in the literature:
\begin{itemize}
\item the \emph{Browkin type} $\tau_B=(\QQ,p,\{- \frac {p-1} 2,\ldots, \frac {p-1} 2\})$  (see \cite{Browkin1978, Bedocchi1988, Bedocchi1989, Bedocchi1990, Browkin2000, CapuanoMurruTerracini2020});
\item the \emph{Ruban type} $\tau_R=(\QQ,p,\{0,\ldots, p-1\})$  (see \cite{Ruban1970, Laohakosol1985, Wang1985, CapuanoVenezianoZannier2019}).
\end{itemize}
\end{example}

\begin{remark} The absolute Galois group $\Gal(\Qalg/\QQ)$ acts on the set of types; indeed, if $\tau=(K,\mathfrak{P}, s)$ is a type, then $\sigma\in \Gal(\Qalg/\QQ)$ induces a continuous map $K_{v_0} \to K^\sigma_{v'_0}$, where $v'_0\in \calM_{K^\sigma}$ corresponds to $\mathfrak{P}^\sigma$. Then  $\tau^\sigma=(K^\sigma,\mathfrak{P}^\sigma, s^\sigma)$ is also a type, where $s^\sigma=\sigma\circ s\circ \sigma^{-1}$. In particular, if $K/\QQ$ is a Galois extension and $\sigma$ belongs to the decomposition group 
$$D_\mathfrak{P}=\{\sigma\in\Gal(\Qalg/\QQ)\ |\  \mathfrak{P}^\sigma=\mathfrak{P}\},$$
   then $\tau^\sigma=(K,\mathfrak{P}, s^\sigma)$ is again a $\mathfrak{P}$-adic type.
\end{remark}

\subsection{ $\mathfrak{P}$-adic continued fractions associated to types}
In this section we give the definition of $\mathfrak{P}$-adic continued fraction algorithm associated to $\mathfrak{P}$-adic types and prove some general properties for these continued fractions generalising the analogous well-established properties in the case of Browkin and Ruban. \\

Let $\tau=(K,\mathfrak{P},s)$ be a type and put 
\[
\mathcal{Y}_s:=im(s).
\]
Then, $\mathcal{Y}_s$ is a discrete subset of $K_{v_0}$.

\begin{definition} Let $\tau=(K,\mathfrak{P},s)$ be a type. 
A \emph{continued fraction}  of type $\tau$ is a possibly infinite sequence
 $$[a_0,a_1,\ldots]$$
 of elements of $\mathcal{Y}_s$ such that $|a_n|_{v_0}>1$ for $n\geq 1$.
\end{definition}
We define the sequences $(A_n)_{n=-1}^\infty$, $(B_n)_{n=-1}^\infty$  by putting
\begin{align*} 
A_{-1}=1,\ &A_0=a_0,\ A_{n}= a_nA_{n-1}+A_{n-2},\\
B_{-1}=0,\ &B_0=1,\ \ B_{n}= a_nB_{n-1}+B_{n-2},
\end{align*}
for $n\geq 1$. By using matrices, we can write 
 \begin{eqnarray}\label{eq:matriciA}
 \mathcal{A}_n &=& \begin{pmatrix} a_n &1 \\
   1 &0 \end{pmatrix}\quad\hbox{ for } n\geq 0,\\
   \mathcal{B}_n &= &\begin{pmatrix} {A_{n}} &{A_{n-1}}\nonumber \\
  {B_{n}} &{B}_{n-1} \end{pmatrix} \quad\hbox{ for } n\geq 0;
 \end{eqnarray}
 then,
 \begin{equation*}\mathcal{B}_n =\mathcal{B}_{n-1}\mathcal{A}_n= \mathcal{A}_0\mathcal{A}_1\ldots \mathcal{A}_n.
   \end{equation*}
 Notice that
 $$ \det(\mathcal{A}_n) = -1,\quad
  \det(\mathcal{B}_n) = (-1)^{n-1}.$$
We define the $n^{th}-$\textit{convergent} to be
\begin{eqnarray*}
 Q_n&= &\frac {A_{n}}{B_{n}}=
a_0+\cfrac{1}{a_1+\cfrac{1}{\ddots+\cfrac{1}{a_n}}}
\quad \hbox{for  $n\geq 0$.}
\end{eqnarray*}
Let $[a_0,a_1,\ldots]$ be a continued fraction of type $\tau$; then, an easy induction shows that 
\begin{equation}\label{eq:K} |B_n|_{v_0}=\prod_{j=1}^n |a_j|_{v_0}.\end{equation}

We notice that the sequence of the convergents $(Q_n)_{n\in \NN}$ converges $\mathfrak{P}$-adically; indeed, it is easy to see that $|Q_n-Q_{n-1}|_{v_0}=\frac 1 {|B_n|_{v_0}|B_{n-1}|_{v_0}}$. Then the claim follows by \eqref{eq:K}, because of the hypothesis that $|a_n|_{v_0}>1$ for every $n\geq 1$. \\

Conversely, every $\alpha\in K_{v_0}$ is the limit of a (unique) continued fraction of type $\tau$  obtained applying the following algorithm:
\begin{equation}
\label{eq:alpharel} 
\left\{ \begin{array}{lll}
\alpha_0&=&\alpha, \\
\alpha_{n+1}&= & \frac 1 {\alpha_{n}-a_n},\\
a_n&=& s(\alpha_n).
\end{array}\right.
\end{equation}
Then, we have the following:
\begin{eqnarray*}	
\alpha_n &=& a_n+\frac 1 {\alpha_{n+1}}.
\end{eqnarray*}
The sequence $[a_0,a_1,\ldots ]$ obtained by applying \eqref{eq:alpharel} is called the \emph{continued fraction expansion of type $\tau$} of $\alpha$.
For any $n\geq -1$, we define 
\begin{equation}
V_n := A_n-\alpha B_n.  \label{eq:Vn}
\end{equation}
The following properties are easily proved (see \cite[Section 2]{CapuanoMurruTerracini2020}).
\begin{proposition}\label{prop:comblinnulla}
For every $n\geq 1$, one has:
\begin{itemize}
\item[a)] $ V_n=a_nV_{n-1}+V_{n-2}$;
\item[b)] $\alpha_{n}V_{n-1}+V_{n-2}=0$;
\item[c)] $|V_n|_{v_0}=\prod_{j=1}^{n+1}\frac 1 {|a_j|_{v_0}}$;
\item[d)] $ V_n=(-1)^{n+1}\prod_{j=1}^{n+1} \frac 1 {\alpha_j};$
\item[e)] $ \alpha_{n} B_{n-1}+ B_{n-2}=\prod_{j=1}^{n} {\alpha_j};$
\item[f)] \label{eq:formulauno} $ \alpha=\frac {\alpha_nA_{n-1}+A_{n-2} }{\alpha_nB_{n-1}+ B_{n-2} };$
\item[g)] $|\alpha_n|_{v_0}=|a_n|_{v_0}> 1$.
\end{itemize}
\end{proposition}

\subsection{The quality of the $\mathfrak{P}$-adic approximation}  
Given $\alpha\in K_{v_0}$, for every $n\ge 1$ let us put $\epsilon_n(\alpha):= {|V_{n-1}|_{v_0}}$, with the convention that $\epsilon_0(\alpha)=1$.
Notice that, for every $i=0,\ldots, n$, we have
\begin{equation}\label{eq:epsilon}
    \epsilon_n(\alpha) = \epsilon_i(\alpha) \epsilon_{n-i}(\alpha_i).
\end{equation}
\begin{proposition}\label{prop:quality1} Given $\alpha,\alpha'\in K_{v_0}$, let $[a_0,\ldots, a_k,\ldots]$, $[a'_0,\ldots, a'_k,\ldots]$ be the continued fraction expansions of type $\tau$ of $\alpha$ and $\alpha'$, respectively. Assume that the length of the expansion of $\alpha$ is bigger or equal than $n$. If $|\alpha-\alpha'|_{v_0}<\epsilon_n(\alpha)^2$, then the length of the expansion of $\alpha'$ is bigger or equal than $n$ and $a_i=a'_i$ for every $i=0,\ldots, n$.
\end{proposition}
\begin{proof}
We argue by induction on $n$. The claim is certainly true for $n=0$, since 
\[
|x-y|_{v_0}<1 \Leftrightarrow s(x)=s(y).
\]
Suppose now that $|\alpha-\alpha'|_{v_0}<\epsilon_{n+1}(\alpha)^2$. First, for $n=0$ we have that $a'_0=a_0$; moreover, we observe that the hypothesis implies $|\alpha-\alpha'|_{v_0}<\frac 1 {|a_1|_{v_0}}=|\alpha-a_0|_{v_0}$.  Applying the properties of the non-archimedean absolute values, we have
$$\frac 1{|a'_1|_{v_0}} =|\alpha'-a_0|_{v_0}=\max\{|\alpha'-\alpha|_{v_0},|\alpha-a_0|_{v_0}\}=|\alpha-a_0|_{v_0}=\frac 1{|a_1|_{v_0}},$$
so that $|a_1|_{v_0}=|a'_1|_{v_0}$.
By \eqref{eq:epsilon} we have
$\epsilon_{n+1}(\alpha)=\frac 1 {|a_1|_{v_0}}\epsilon_n(\alpha_1),$
hence
\begin{equation*} |\alpha_1-\alpha'_1|_{v_0} = \left| \frac 1 {\alpha-a_0} -\frac 1 {\alpha'-a_0}\right |_{v_0}
=|a_1|_{v_0}^2 |\alpha-\alpha'|_{v_0}
 <\prod_{j=2}^{n+1} \frac 1 {|a_j|_{v_0}^2}=
\epsilon_n(\alpha_1)^2.
\end{equation*}
Applying the inductive hypothesis, we have that $a_i=a'_i$ for $i=1,\ldots, n+1$, concluding the proof.
\end{proof}

The next proposition proves that, if the $n^{th}$-convergents of $\alpha$ and $\beta$ are the same, then the two numbers are $v_0$-adically close. More precisely we have the following result: 
\begin{proposition}\label{prop:dueenne}
Let $e$ be the ramification index of $K_{v_0}/\QQ_p$. Assume $\alpha,\beta\in K_{v_0}$ be such that $Q_n^\alpha=Q_n^\beta$. Then $|\alpha-\beta|_{v_0}<\frac 1 {p^{\frac {2n} e}}$.
\end{proposition}
\begin{proof}
 Notice that the hypothesis $Q_n^\alpha=Q_n^\beta$ is equivalent to say that the first $n+1$ partial quotients of $\alpha$ and $\beta$ are equal. We argue by induction on $n$. The claim is certainly true for $n=0$, so assume that $n\geq 1$ and $Q_n^\alpha=Q_n^\beta$. This implies that $a_0=s(\alpha)=s(\beta)$ and $Q_{n-1}^{\alpha_1}=Q_{n-1}^{\beta_1}$, so that $|\alpha_1-\beta_1|_{v_0}<\frac 1 {p^{\frac {2(n-1)} e}}$. Since $|\alpha_1\beta_1|_{v_0}\geq p^{\frac 2 e}$, we have
 $$|\alpha-\beta|_{v_0}=\left |a_0+\frac 1{\alpha_1}-a_0 - \frac 1 {\beta_1}\right |_{v_0}=\frac{|\alpha_1-\beta_1|_{v_0}} {|\alpha_1|_{v_0}|\beta_1|_{v_0}}< \frac 1 {
 p^{\frac {2(n-1)+2} e}}= \frac 1 {p^{\frac {2n}e}},$$
 proving the claim.
\end{proof}

\section{Finiteness and periodicity}
Let $\tau=(K,\mathfrak{P},s)$ be a type. It is clear by the construction that every finite continued fraction expansion of type $\tau$ represents an element in $K$, but the converse is not true in general. Motivated by a question of Rosen in the archimedean case, we are interested in giving necessary and sufficient conditions such that a type satisfies the property that every element of $K$ has finite continued fraction expansion of type $\tau$.

Following \cite{MVV2020}, where the authors adress Rosen problem in the archimedean case, we introduce the following definitions.
\begin{definition}
\begin{itemize} \item[a)]  
We say that $\tau$ satisfies the  \emph{Continued Fraction Finiteness  property ($\CFF$)}  (resp. the \emph{Continued Fraction Periodicity  property ($\CFP$)}) if every $\alpha\in K$ has a finite (resp. finite or periodic) $\tau$-expansion. 
\item[b)] We say that the field $K$ satisfies the \emph{$\mathfrak{P}$-adic Continued Fraction Finiteness  property ($\CFF$)} (resp. the \emph{$\mathfrak{P}$-adic Continued Fraction Periodicity property ($\CFP$)}) if there is a type $\tau=(K,\mathfrak{P},s)$  satisfying the  $\CFF$ (resp. $\CFP$) property.
\end{itemize}
\end{definition}
In the next sections, we show that the Browkin and Ruban continued fractions satisfy the $\CFF$ and $\CFP$ property, respectively. 

\subsection{Browkin expansion of rational numbers}
In \cite[\S 3]{Browkin1978} Browkin proved that, for every odd prime $p$, every rational number has a finite Browkin $p$-adic continued fraction expansion. Although not explicitly stated, the proof also gives a quantitative estimate for the length $n$ of the expansion; namely, if $\alpha=p^s\frac x y$ with $s,x,y\in\ZZ$ and $(x,y)=1$, then $n$ does not exceed $|x|_\infty+2|y|_\infty+1$. We present here a slightly different proof improving the bound for $n$.
\begin{lemma}\label{lem:lemmaapprox2} Let $(t_n)_{n\in \NN} $ be a sequence of real numbers $\geq 0$, such that there exist   $c_0, c_1\in\mathbb{R}_{>0}$   satisfying 
$$t_{n+2} < c_{1}t_{n+1}+ c_{0}t_n, $$
and let $\tilde x$ be the (unique) positive real root of the polynomial
\begin{equation*} f(X)=X^{2}-c_1X-c_0. \end{equation*}
Then, 
\begin{itemize}
    \item[a)] $ |t_n|_\infty \leq \max\{t_0, \frac{t_1}{\tilde x}\}\cdot \tilde x^n$  for every $n\in\NN$;
\item[b)] if $c_0+c_1<1$, we have that $|t_n|_\infty \to 0.$
\end{itemize}
\end{lemma}
\begin{proof} The first assertion follows easily applying an induction on $n$. To prove part $b)$, first notice that, since $\tilde x>0$ and 
$$\tilde x =c_1 +\frac  {c_{0}}{\tilde x},$$
then $c_1<\tilde x$. If $c_0+c_1<1$, then $f(1)=1-(c_0+c_1)>0$, so that $0<\tilde x <1$, and we can conclude that $c_1<\tilde x<1$. Therefore,
\begin{equation*}\tilde x <1 \Leftrightarrow c_1+c_0 <1. \end{equation*}
Then, the claim follows from part $a)$.\end{proof}

\begin{proposition}\label{prop:Browkinfin}
Let $\alpha \in\QQ$, and write
$\alpha=\frac{x_0}{y_0}$,
with $y_0\in\ZZ$ not divisible by $p$ and $x_0\in\ZZ\left[\frac 1 p\right]$. Then, the Browkin $p$-adic continued fraction expansion of $\alpha$ is finite, and its length is strictly bounded by $$\frac{\log(\max\{|x_0|_\infty, |y_0|_\infty\})}{\log\left(\frac{p(\sqrt{p^2+16}-p)}{4}\right)}.$$
\end{proposition}
\begin{proof}  Consider the sequence $(V_n)_{n\in \NN}$ defined by \eqref{eq:Vn} and put $y_{n+1}=y_0V_n$. By Proposition \ref{prop:comblinnulla} $c)$ we have $y_n\in\ZZ[\frac 1 p]\cap p^n\ZZ_p=p^{n}\ZZ$. Moreover, by Proposition \ref{prop:comblinnulla} $a$) the $y_n$'s satisfy the recurrence 
$$ y_{n+1}=a_{n}y_{n}+y_{n-1}, $$ 
with $|a_n|_\infty <\frac p 2$. Then, $\frac {y_n}{p^n}\in\ZZ$ 
and, applying Lemma \ref{lem:lemmaapprox2}, we have
$$
    \frac {|y_{n}|_\infty } {p^{n}} < \frac 1 2 \frac {|y_{n-1}|_\infty } {p^{n-1}} + \frac 1 {p^2} \frac {|y_{n-2}|_\infty } {p^{n-2}}\\
 < M \tilde x^n,$$
  where  $M=\max\left\{|y_{0}|_\infty,\frac {|y_{1}|_\infty } {p} \right\}$ and $\tilde x=\frac 1 {4p} (\sqrt{p^2+16}+p)$ is the positive root of the polynomial $X^2-\frac 1 2 X- \frac 1 {p^2}$.
It follows that $y_{n}=0$ for $\tilde x^n \leq \frac 1 {M}$.
Furthermore we have
$$
     \frac{|y_{1}|_\infty} {p}= \frac 1 p |x_0-a_0y_0|_\infty < \frac 1 p {|x_0|_\infty} + \frac 1 2  |y_0|_\infty,
    $$
therefore 
$$ M<\max\left \{ |y_0|_\infty, \frac 12 (|x_0|_\infty+|y_0|_\infty)\right \}\leq \max\{ |x_0|_\infty,|y_0|_\infty
\},$$ 
so that 
     $y_{n}=0$ for $\tilde x^n \leq \frac 1 {\max\{ |x_0|_\infty,|y_0|_\infty
\}}$, that is for $ n \geq  -\frac{\log(\max\{ |x_0|_\infty,|y_0|_\infty
\})}{\log(\tilde x^{-1})}$, as wanted.
    \end{proof}

\subsection{Ruban expansion of rational numbers}
In \cite{Ruban1970}, Ruban introduced a $p$--adic continued fraction corresponding to the type $\tau_R=(\QQ,p,\{0,\ldots, p-1\})$, proving that the continued fraction expansions coming from this type enjoy nice ergodic properties. However, it is easy to see that the Ruban type cannot satisfy $\CFF$, since negative rational numbers cannot have a terminating Ruban continued fraction. In this setting, Laohakosol \cite{Laohakosol1985} and, independently, Wang \cite{Wang1985} proved that $\tau_R$ satisfies $\CFP$, proving in particular that, if a rational number as non-terminating $\tau_R$-expansion, then the tail is equal to $\left[ \overline{1-\frac{1}{p}}\right]$. However, none of these arguments were effective; more recently, in \cite{CapuanoVenezianoZannier2019} the authors gave a quantitative estimate for the length of the expansion in case this is finite, and an estimate on the length of the pre-periodic part in terms of the height of the rational number.

\subsection{The main tools}
For $x\in\CC$, we define
$$\theta(x)=\frac 1 2(|x|_\infty +\sqrt{|x|_\infty^2+4});$$
then, we have the following inequality: 
$$|x|_\infty\leq \theta(x)\leq |x|_\infty+1, $$
and the map $\theta$ is a bijection from $[0,+\infty)$ to $[1,+\infty)$ whose inverse is given by $y\mapsto y-\frac 1 y$. In this section we will prove that, given a type $(K, \mathfrak{P}, \tau)$, a suitable bound involving the values of $\theta$ on the elements in the image of $s$ and on their conjugates will guarantee that the type satisfies the CFF (resp. CFP) property.

We first need the following lemma:
\begin{lemma}\label{lem:matrixnorm} Let $(a_n)_{n\ge 0}$ be any sequence of complex numbers and let $(v_n)_{n \ge -1}$ be a sequence of complex numbers satisfying, for every $n\ge 2$ the recurrence formula:
$$v_n=a_nv_{n-1}+v_{n-2}.$$
Then for every $n\geq 0$,
$$\max\{|v_n|_\infty, |v_{n-1}|_\infty\}\leq \sqrt{|v_n|^2_\infty+|v_{n-1}|_\infty^2} \leq \sqrt{|v_0|^2_\infty+|v_{-1}|_\infty^2}\cdot \prod_{j=1}^n \theta(a_j) .$$
\end{lemma}
\begin{proof}
For any complex matrix $M$, let us consider the operator norm 
\begin{align*}
||M|| &= \sup_{{\bf v}\not={\bf 0}} \frac {||M{\bf v}||}{||{\bf v}||}
\end{align*}
where $||{\bf v}||$ denotes the Euclidean norm of a complex vector. The following facts are well known (see for example \cite[Chapter 5]{Horn2013}):
\begin{itemize}
    \item $||M_1\cdot M_2||\leq ||M_1||\cdot ||M_2||$;
    \item  $||M||=\sqrt{|\gamma|_\infty}$, where $\gamma$ is the dominant eigenvalue of $M\cdot M^*$ (here $M^*$ denotes the transpose conjugate of $M$).
\end{itemize}
In particular we see that, for every $a\in\CC$,
\begin{equation*}
\left |\left | \begin{pmatrix} a & 1\\ 1 & 0\end{pmatrix} \right |\right |=\theta(a).\end{equation*} 
 Let $\mathcal{A}_n$ be the matrix defined as in \eqref{eq:matriciA};
then, for every $n\geq 1$,
$$ \left|\left | \begin{pmatrix} v_n\\v_{n-1}\end{pmatrix} \right|\right |= \left|\left |\mathcal{A}_n\begin{pmatrix} v_{n-1}\\v_{n-2}\end{pmatrix}\right|\right |\leq \left|\left |\mathcal{A}_n \right |\right| \left|\left | \begin{pmatrix} v_{n-1}\\v_{n-2}\end{pmatrix}\right|\right|=\theta(a_n)\left|\left | \begin{pmatrix} v_{n-1}\\v_{n-2}\end{pmatrix}\right|\right|,$$
so that 
$$ \max\{|v_n|_\infty, |v_{n-1}|_\infty\}\leq \left|\left | \begin{pmatrix} v_n\\v_{n-1}\end{pmatrix} \right|\right |\leq \left|\left | \begin{pmatrix} v_0\\v_{-1}\end{pmatrix} \right|\right |\cdot \prod_{j=1}^n \theta(a_j).$$
\end{proof}

\begin{theorem}\label{teo:FCcondition}
Let $\tau=(K,\mathfrak{P},s)$ be a type. Let $\Sigma$  be the set of embeddings of $K$ in $\CC$, and let us denote by 
$$\nu_\tau=\sup\left\{\frac{\prod_{\sigma\in\Sigma}\theta(a^\sigma)} { |a|_{v_0}^{d_{v_0}}} \ |\ a\in\calY_s, |a|_{v_0}>1\right \}. $$
Then, \begin{itemize}
    \item[a)] if $\nu_\tau\leq 1$, then $\tau$ satisfies $\CFP$;
    \item[b)] if $\nu_\tau< 1$, then then $\tau$ satisfies $\CFF$. 
\end{itemize}
\end{theorem}

\begin{proof}
Let $\alpha\in K$; we notice that, for every $\sigma\in\Sigma$, the sequence $V_n^{\sigma}$ satisfies the recurrence formula
$$V_n^{\sigma}=a_n^\sigma V_{n-1}^\sigma +V_{n-2}^\sigma,$$
for every $n\ge 1$, so by Lemma \ref{lem:matrixnorm}, there exists a suitable $C>0$ (depending on $\alpha$) such that
\begin{equation} \label{eq:theta_est}
\prod_{\sigma\in\Sigma}\sup\{|V_n^\sigma|_\infty, |V_{n-1}^\sigma|_\infty\}\leq C\prod_{j=1}^n \prod_{\sigma\in \Sigma} \theta(a_j^\sigma).
\end{equation}
Since $V_n\alpha_{n+1}+V_{n-1}=0$, and recalling that $V_n\in K$, we have that
\begin{align*}
    H(\alpha_{n+1})^d & = 
    H\left(-\frac{V_{n-1}}{V_n}\right)^d\\
    &= \prod_{v\in \calM_K}\sup\left \{\left |\frac{V_{n-1}}{V_n}\right |_v^{d_v},1\right \},\\
    &=\prod_{v\in \calM_K}\sup\{|V_n|_v^{d_v},|V_{n-1}|_v^{d_v}\}.
\end{align*}
Using that, for every non archimedean place $v\not=v_0$ we have that 
\[
|V_n|_v\leq \max\{|A_n|_v,|B_n|_v|\alpha|_v\}\leq \max\{|\alpha|_v,1\},
\]
and that $|V_n|_{v_0}<|V_{n-1}|_{v_0}$, we can bound from above the previous quantity obtaining
\begin{align*}
        H(\alpha_{n+1})^d&\leq H(\alpha)^d\cdot  |V_{n-1}|_{v_0}^{d_{v_0}}  \prod_{\sigma \in \Sigma} \max\{|V^\sigma _n|_\infty,|V_{n-1}^\sigma|_\infty \},\\
    \intertext{and, applying \eqref{eq:theta_est}, we have}
    &\leq C\cdot H(\alpha)^d \cdot |V_{n-1}|_{v_0}^{d_{v_0}} \prod_{j=1}^n \prod_{\sigma\in \Sigma} \theta(a_j^\sigma)\\
    &\leq C\cdot H(\alpha)^d\cdot |V_{n-1}|_{v_0}^{d_{v_0}} \nu_\tau^n  \prod_{j=1}^n|a_j|_{v_0}^{d_{v_0}}.
    \end{align*}
    Since $V_{n-1}=(-1)^n\prod_{j=1}^{n} \frac 1{\alpha_j} $ by Proposition \ref{prop:comblinnulla}, we have
    $|V_{n-1}|_{v_0}= \prod_{j=1}^{n} \left |\frac 1{a_j}\right |_{v_0} $, therefore
    $$ H(\alpha_{n+1})^d \leq C\,\nu_\tau^n \cdot H(\alpha)^d.$$
    Let us suppose that $\nu_\tau\leq 1$; by Northcott finiteness theorem, the $\alpha_n$'s vary in a finite set, so that either the continued fraction expansion of type $\tau$ for $\alpha$ is finite or there exist $m,n\in\NN$ such that $\alpha_m=\alpha_n$, so that the expansion is periodic. This proves part a). As for b), we see that if   $\nu_\tau< 1$ then  either $\alpha$ has a finite expansion or $H(\alpha_n)$ becomes eventually zero, and the latter is a contradiction.\end{proof}

\subsection{A particular case: the $\CFF$ and $\CFP$ properties for special types}

When the type $\tau=(K, \pi, \mathcal R)$ is special, i.e. the floor function is essentially given by the choice of a generator of $\mathfrak{P}$ and by the choice of a set of representatives $\calR$ as explained in Section \ref{sec:special_type}, Theorem \ref{teo:FCcondition} gives a more explicit criterion to detect $\CFP$ and $\CFF$ properties, which is the following result.
\begin{theorem}\label{cor:FEC} Let $\tau=(K,\pi,\mathcal{R})$ be a special type, and let $\Sigma$  be the set of embeddings of $K$ in $\CC$. For every $\sigma\in\Sigma$, let $L_{\sigma}=\max\{|c^\sigma|_\infty \ |\ c\in\mathcal{R}\}$, and $\lambda_{\sigma}=|\pi^\sigma|_\infty$.
Assume that, for every $\sigma\in \Sigma$, 
$$\lambda_{\sigma}>1 \hbox{ and } L_\sigma\leq (\lambda_\sigma-1)\left(1-\frac 1 {\lambda_\sigma^2}\right);  $$
then,
\begin{itemize} 
\item[a)]  $\tau$ satisfies the $\CFP$ property;
\item[b)] if moreover
$L_\sigma<(\lambda_\sigma-1)\left(1-\frac 1 {\lambda_\sigma^2}\right)$ for at least one $\sigma$, 
then $\tau$ satifies the $\CFF$ property.
\end{itemize}
\end{theorem}

\begin{proof} 
Recall that, for a special type $\tau=(K, \pi, \mathcal R)$,  every $a\in \calR$ has the form $\sum_{j=-k}^0 c_j
\pi^j$. Then, for every $a\in \calR$ and every $\sigma\in \Sigma$, since by hypothesis $\lambda_{\sigma}>1$, we have
\begin{equation*} \label{eq:boundY} 
|a^\sigma|_\infty \leq L_\sigma \sum_{j=0}^k \frac{1}{\lambda_\sigma^j}
 \leq \frac{L_\sigma\lambda_\sigma}{\lambda_\sigma-1},
 \end{equation*}
 which using the hypothesis, gives $|a^\sigma|_\infty\leq \lambda_\sigma - \frac{1}{\lambda_\sigma}$, so that  $\theta(a^\sigma)\leq \lambda_\sigma$. Then, under the assumption $|a|_{v_0}>1$, we have
\begin{align*}\frac{\prod_{\sigma\in\Sigma} \theta(a^\sigma)}{|a|_{v_0}^{d_{v_0}}}& \leq \frac{\prod_{\sigma\in\Sigma} \lambda_\sigma}{|a|_{v_0}^{d_{v_0}}}\leq |N_{K/\QQ}(\pi)|_\infty \cdot |\pi|_{v_0}^{d_{v_0}}\\
\intertext{and, writing $d_{v_0}=e_{v_0}f_{v_0}$, where $e_{v_0}$ is the ramification index and $f_{v_0}$ the residual degree at $v_0$, we further have}
&\leq p^{f_{v_0}}\cdot \frac 1 {p^{\frac {d_{v_0}}{e_{v_0}}}}\leq 1,
\end{align*}
hence $\nu_\tau\leq 1$ and we can apply part $a$) of Theorem \ref{teo:FCcondition}, which proves $a)$. 

Let us now assume that there exists an embedding $\sigma_0 \in \Sigma$ such that $L_{\sigma_0}<(\lambda_{\sigma_0}-1)\left(1-\frac 1 {\lambda_{\sigma_0}^2}\right)$; 
then, $|a^{\sigma_0}|_\infty \leq  \frac {L_{\sigma_0}\lambda_{\sigma_0} }{\lambda_{\sigma_0}-1} <\lambda_{\sigma_0}-\frac 1 {\lambda_{\sigma_0}}$, hence
$$\theta(a^{\sigma_0}) \leq \theta \left (\frac {L_{\sigma_0}\lambda_{\sigma_0} }{\lambda_{\sigma_0}-1}\right ) < \lambda_{\sigma_0}.$$ 
Let us put $J:= \frac 1 {\lambda_{\sigma_0}}\cdot \theta \left (\frac {L_{\sigma_0}\lambda_{\sigma_0} }{\lambda_{\sigma_0}-1}\right )<1$. Then, the same calculation as before shows that 
$$\frac{\prod_{\sigma\in\Sigma} \theta(a^\sigma)}{|a|_{v_0}^{d_{v_0}}} \leq J\cdot  \frac{\prod_{\sigma\in\Sigma} \lambda_\sigma}{|a|_{v_0}^{d_{v_0}}} \leq J<1,$$ 
so that $\nu_\tau<1$ and the conclusion follows by applying part $b$) of Theorem \ref{teo:FCcondition}.
\end{proof} 

\section{The $\CFF$ property for norm Euclidean number fields} \label{sec:5}
Let $K$ be a number field of degree $d$ and let us denote by $r_1,r_2$ the number of real and complex embeddings of $K$. For every $1\leq i\leq r_1$ let $\sigma_i$ be the real embeddings and, for every $1\leq j\leq r_2$, let $(\tau_j, \overline{\tau}_j)$ be the $r_2$ pairs of complex embeddings. We denote by $\Sigma$ be the whole set of embeddings.

In what follows, we shall denote by $|\cdot |$ the standard complex absolute value.\\
Let
\begin{align*} i: K& \longrightarrow  \RR^{r_1}\times \CC^{r_2}\\
\lambda &\longmapsto (\sigma_1(\lambda),\ldots,\sigma_{r_{1}}(\lambda),\tau_1(\lambda),\ldots,\tau_{r_2}(\lambda))
\end{align*} be the canonical embedding of $K$, and
$$\ell: K^\times \to \RR^{r_1+r_2}$$ be the logarithmic embedding, i.e., the composition $L\circ i$ where 
\begin{align*} L: \RR^{r_1}\times \CC^{r_2}& \longrightarrow  \RR^{r_1}\times \RR^{r_2}\\
(x_1,\ldots, x_{r_1}, y_1,\ldots, y_{r_2})  &\longmapsto (\log(|x_1|),\ldots, \log(|x_{r_1}|), 2\log(|y_1|),\ldots, 2\log(|y_{r_2}|)).
\end{align*} \\
For $\mathbf{x}=(x_1,\ldots, x_{r_1},y_1,\ldots, y_{r_2})\in \RR^{r_1}\times \CC^{r_2}$, let us define
$$N(\mathbf{x})=\prod_{i=1}^{r_1}|x_i|\cdot \prod_{j=1}^{r_2}|y_j|^2;$$ 
then, $N(i(a))=|N_{K/\QQ}(a)|$ for every $a\in K.$
\medskip

In what follows, we want to deepen the study of types in the case of norm Euclidean number fields; in this setting, we will see that the existence of $\mathfrak{P}$-adic types satisfying CFF is related to the notion of Euclidean minimum of the field.
\subsection{Euclidean minimum} We recall the definition and main properties of the Euclidean minimum.
\begin{definition} Let $\alpha\in K$, the \emph{Euclidean minimum} of $\alpha$ is the real number
$$m_K(\alpha)=\inf\left\{|N_{K/\QQ}(\alpha -\gamma)|\ |\ \gamma \in \calo_K\right\}.$$
\end{definition}
The Euclidean minimum can be extended to $K\otimes_\QQ \RR\simeq \RR^{r_1}\times \CC^{r_2} $ as in the following definition.
\begin{definition}
Let $\mathbf{x} \in  \RR^{r_1}\times \CC^{r_2}$.  The \emph{inhomogeneous minimum} of $\mathbf{x}$ is the real number
$$\overline{m}_{K}(\mathbf{x} )=\inf\left\{|N(\mathbf{x}  -i(\gamma))|\ |\ \gamma \in \calo_K\right\}.$$
\end{definition}
It is clear that $m_K(\alpha)=\overline{m}_K(i(\alpha))$ for every $\alpha\in K$.
Moreover, $\overline{m}_K$ induces an upper semi-continuous map of the torus $\RR^{r_1}\times \CC^{r_2}/i(\calo_K)$, which is a compact set; therefore, $\overline{m}_K$ is bounded and attains its maximum. Then, also $m_K$ is bounded on $K$, and we can give the following definitions:
\begin{definition}
    The \emph{inhomogeneous minimum} of $K$ is the positive real number
    $$\overline{M}(K)=\sup\{\overline{m}_K(\mathbf{x})\ |\ \mathbf{x}\in \RR^{r_1}\times \CC^{r_2}\}.$$
    The \emph{Euclidean minimum} of $K$ is the positive real number
    $$M(K)=\sup\{m_K(\alpha)\ |\ \alpha\in K\}.$$
\end{definition}

By the above definition it is clear that $M(K)\leq \overline{M}(K)$. Moreover, 
it is easy to see that $K$ is norm Euclidean if and only if $m_K(\alpha)<1$ for every $\alpha\in K$. Therefore $K$ is norm-Euclidean if $M(K)<1$ and it is not norm-Euclidean if $M(K)>1$. The following non trivial result holds:
\begin{theorem}{\cite[Theorem 3]{Cerri2006}}\label{teo:cerri1}
If $K$ is a number field, then $M(K)=\overline{M}(K)$. 
\end{theorem}
\subsection{The $\CFF$ property for norm Euclidean number fields}

 In order to prove our main results about the $\CFF$ property for norm Euclidean fields, we shall impose conditions on a $\mathfrak{P}$-adic type $\tau$ allowing to suitably bound the quantity $\nu_\tau$ defined in Theorem \ref{teo:FCcondition}. In particular, to apply Theorem \ref{teo:FCcondition} we need to control the size of 
$\prod_{\sigma\in \Sigma} \theta(a^\sigma)$ for $a\in\calY_\tau$ such that $|a|_{v_0}>1$. 

First, let us notice that
$$ |N_{K/\QQ}(a)|<  \prod_{\sigma\in \Sigma}\theta(a^\sigma)< \prod_{\sigma\in \Sigma} ( |a^\sigma|+1),$$ 
and we can regard the latter expression as a sum 
$$\prod_{\sigma\in \Sigma} ( |a^\sigma|+1) =|N_{K/\QQ}(a)| + F(|a^\sigma|,\sigma\in\Sigma). $$

We shall exploit norm Euclidean properties of the field to bound $|N_{K/\QQ}(a)|$; on the other hand, in order to bound the second addend $F$, we need  more refined conditions allowing to control each archimedean absolute value of $a$ and not only their product. To this purpose, the following lemma will be useful: 
\begin{lemma}\label{lem:univin}
There exists $T_0>0$ (depending only on $K$) such that, for every $a\in K^\times $, there exists $u\in \calo_K^\times$ satisfying, for every $\sigma\in\Sigma $,
$$ |(au)^\sigma|  \leq T_0 \sqrt[d]{|N_{K/\QQ}(a)|}. $$
\end{lemma}
\begin{proof} We denote by $\mathcal H$ the hyperplane in $\RR^{r_1}\times \RR^{r_2}$ defined by 
$$x'_1+...+x'_{r_1}+2y'_1+...+2y'_{r_2}=0, $$
and we denote as before by $\ell:K^{\times} \rightarrow \RR^{r_1+r_2}$ the logarithmic embedding.
Since $\ell(\calo_K^\times)$ is a lattice in $\mathcal H$, there exists $T>0$ such that, for every $\mathbf{b} \in \mathcal H$ there is $u\in \calo_K^\times$ with $||\mathbf{b}+\ell(u)||_{\infty}<T$, where $||\cdot||_{\infty}$ is the sup norm in $\RR^{r_1+r_2}$. 

For $a\in K^\times$, let us take $$\mathbf{b}=L \left(\frac {i(a)}{\sqrt[d]{|N_{K/\QQ}(a)|}}\right )=\ell(a)-\frac {\log({|N_{K/\QQ}(a)|})} d (1,\ldots,1) ;$$ by construction, $\mathbf{b}\in \mathcal H$, hence there exists $u\in \calo_K^\times$ such that $||\mathbf{b}+\ell(u)||_\infty<T$.  This implies that, for every $\sigma\in \Sigma$,
$$
\left | \log \left (\left |\frac  {(au)^\sigma}{\sqrt[d]{|N_{K/\QQ}(a)|}}\right | \right ) \right |=
\left | \log \left (\left |\frac {a^\sigma}{\sqrt[d]{|N_{K/\QQ}(a)|}}\right | \right ) +\log(|u^\sigma|) \right | \leq ||\mathbf{b}+\ell(u)||_\infty <T
$$
so that $\left |{(au)^\sigma}\right |<T_0 \sqrt[d]{|N_{K/\QQ}(a)|}$ for a suitable $T_0$, as wanted.
\end{proof}

For $\alpha\in\calo_K$, $\epsilon>0$ and $\mathbf{x}\in \RR^{r_1}\times \CC^{r_2}$, we define
$$U_\epsilon(\alpha)=\{\mathbf{y}\in \RR^{r_1}\times \CC^{r_2}\ |\ |N(i(\alpha)-\mathbf{y})| <\epsilon\}.$$
Then $U_\epsilon(\alpha)$ is an open subset of $\RR^{r_1}\times \CC^{r_2}$. 
\begin{theorem}\label{teo:euclidean}
Assume that $K$ is a norm Euclidean number field such that $M(K)<1$.
Then, the field $K$ satisfies the $\mathfrak{P}$-adic $\CFF$-property for all but finitely many prime ideals $\mathfrak{P}$ of $\calo_K$. 
\end{theorem}

\begin{proof} The image of $\calo_K$ in $\RR^{r_1}\times \CC^{r_2}\simeq \RR^d$ is a lattice; we denote by $\calD\subseteq \RR^{r_1}\times \CC^{r_2}$ a compact fundamental domain for  $i(\calo_K)$.
Since by assumption the Euclidean minimum $M(K)<1$, we have by Theorem \ref{teo:cerri1}, that $\calD$ is covered by open neighbourhood of radius $1$, i.e.
$\calD\subseteq \bigcup_{\gamma\in\calo_K} U_1(\gamma)$. Since $\calD$ is compact and the $U_{1}(\alpha)$'s are open, there exists a finite number of elements $\alpha_1,\ldots,\alpha_s$  of $\calo_K$ and a real $0<\tilde{\epsilon}<1$ such that
$\calD\subseteq \bigcup_{i=1}^s U_{\tilde{\epsilon}}(\alpha_i)$. Let $\mathfrak{P}$ be a prime ideal of $\calo_K$; we put $q=N_{K/\QQ}(\mathfrak{P})=\left | \calo_K/\mathfrak{P}\right |$, and, by Lemma \ref{lem:univin}, we choose a generator $\pi$ of $\mathfrak{P}$ such that $|\pi^\sigma| < T_0\sqrt[d]{q}$ for every embedding $\sigma \in \Sigma$.

We define a $\mathfrak{P}$-adic floor function $s$ as follows: let us consider a non trivial coset $\alpha+\mathfrak{P}\calo_{v_0} \subseteq  K_{v_0}$; by strong approximation, it contains an element $\alpha'\in K$ such that $|\alpha'|_v\leq 1$ for every non archimedean $v\in \calM_K$ and $v\not= v_0$. Then, $\alpha'\in\calo_K[\frac 1 \pi]$.  By translating  $\frac{\alpha'} \pi$ by a suitable element of $\calo_K$, we find a $\beta\in\calo_K[\frac 1 \pi]$ such that $i(\beta)\in \calD$ and $\alpha'\equiv \pi\beta \pmod{\mathfrak{P}}$. Since $\calD$ is covered by $\bigcup_{i=1}^s U_{\tilde{\epsilon}}(\alpha_i)$, there exists $\alpha_i\in\calo_K$ such that $N_{K/\QQ}(\beta-\alpha_i)<\tilde{\epsilon}$. Then, for every $\gamma\in \alpha+\mathfrak{P}\calo_{v_0}$ we put  $s(\gamma):=\pi(\beta-\alpha_i)$. 

We want now to apply Theorem \ref{teo:FCcondition} to show that the type associated to this floor function satisfies $\CFF$ property for all but finitely many prime ideals $\mathfrak{P}$ of $\mathcal{O}_K$. To prove this, let us call $a=\pi(\beta-\alpha_i)$; then, $N_{K/\QQ}(a)\leq \tilde{\epsilon} q$. Since $i(\beta)\in \calD$ which is a compact set and the $\alpha_i$ are finitely many, there exists $H>0$ depending only on $K$ such that $|\beta^\sigma -\alpha_i^\sigma | < H$ for every $\sigma\in \Sigma$, so that $|a^\sigma| < H |\pi^\sigma| <HT_0\sqrt[d]{q}$.
 It follows that, for every subset $S\subsetneq \Sigma$, 
$$
\prod_{\sigma\in S} |a^\sigma|\leq (HT_0)^{|S|} \sqrt[d]{q^{|S|}}.
$$
Therefore,
\begin{align*} \prod_{\sigma\in\Sigma}\theta(a^\sigma)&< \prod_{\sigma\in\Sigma }(1+|a^\sigma|)
= \sum_{S\subseteq \Sigma} \prod_{\sigma\in S} |a^\sigma| =N_{K/\QQ}(a)+ \sum_{S\subsetneq \Sigma} \prod_{\sigma\in S} |a^\sigma|\\
&\leq \tilde{\epsilon} q + \sum_{S\subsetneq \Sigma}(HT_0)^{|S|} \sqrt[d]{q^{|S|}}\leq \tilde{\epsilon} q + H_1 \sqrt[d]{q^{d-1}}
\intertext{for a suitable $H_1$ depending only on $K$, hence}
\prod_{\sigma\in\Sigma}\theta(a^\sigma) & < \epsilon ' q \hbox{\quad for $q\gg 0$,}
\end{align*}
for a suitable $\epsilon'<1$.\\
Since $|a|_{v_0}^{d_{v_0}}\geq |\frac 1 \pi|_{v_0}^{d_{v_0}}=|N_{K/\QQ}(\pi)|^{-1}_p=q$, we find that $\nu_\tau <1$ for $p\gg 0$. Then the claim follows from Theorem \ref{teo:FCcondition}.
\end{proof}

\begin{remark} We point out that the condition $M(K)<1$ in Theorem \ref{teo:euclidean} is verified for \lq\lq almost all\rq\rq\  norm Euclidean number fields. Indeed, in \cite[Corollary 2]{Cerri2006} it is shown that a norm Euclidean number field with $M(K)=1$ must have unit rank $r=r_1+r_2-1\leq 1$. Since $d=r_1+2r_2$, the only exceptions can occur in the following three cases:
\begin{itemize}
    \item[a)] $K$ is quadratic;
    \item[b)] $K$ is cubic with negative discriminant;
    \item[c)] $K$ is a totally complex quartic field.
\end{itemize} All quadratic norm Euclidean fields have $M(K)<1$: this is easily checked for imaginary quadratic fields;  for real quadratic fields, it is shown in \cite[\S 1.3]{Lezowski2014} that the only $K$ such that $M(K)=1$ is $\QQ(\sqrt{65})$, which is not Euclidean.
On the other hand, Davenport proved in \cite{Davenport1950,Davenport1951} that norm Euclidean fields satisfying $b)$ or $c)$ are finitely many.
\end{remark}

\section{Some more effective results for quadratic fields}

In the case of quadratic fields, we can give more explicit results, proving effective bounds on the prime ideals $\mathfrak{P}$ of $K$ such that $K$ satisfies the $\mathfrak{P}$-adic CFF property. Moreover, in some cases we will show how to construct explicit examples of types satisfying the CFF property. We start our investigation with the case of imaginary quadratic fields.

\subsection{Imaginary norm Euclidean quadratic fields} Let $K=\mathbb{Q}(\sqrt{-D})$ with $D$ a square free integer $>0$. It is known that
$$M(K)= \left\{\begin{array}{ll} \frac {D+1} 4 & \hbox{ if } D\equiv 1,2\pmod 4\\
\frac{ (D+1)^2}{16 D}  &\hbox{ if } D\equiv 3\pmod 4    \end{array}\right. $$
(see for example \cite[Prop. 4.2]{Lemmermeyer1995}).
It follows that the only norm Euclidean quadratic imaginary fields are $\mathbb{Q}(\sqrt{-D})$, with $D=1, 2, 3, 7, 11$ and $M(K)<1$ in each of these cases.  
\begin{proposition}
Let $K=\QQ(\sqrt{-D})$ be a imaginary quadratic norm Euclidean field. Let $\mathfrak{P}$ be a prime ideal of $\calo_K$ with odd residual characteristics. Put $\lambda=\sqrt{N_{K/\QQ}(\mathfrak{P})}$. Then
\begin{itemize}
    \item[a)] if $\sqrt{M(K)}< 1-\frac 1 {\lambda^2} $, then $K$ satisfies the $\mathfrak{P}$-adic $\CFF$ property.
     \item[b)] if $\sqrt{M(K)}< \left(1-\frac 1 {\lambda}\right)^2\left(1+\frac 1 {\lambda}\right)  $, then there exists a special type $\tau=(K,\pi,\mathcal{R})$ satisfying the $\CFF$ property.
\end{itemize}
\end{proposition}

\begin{proof} Let $\mathfrak{P}$ be a prime ideal of $\calo_K$ with odd residual characteristic $p$, and let $\pi$ be a generator of $\mathfrak{P}$. First, notice that $N_{K/\QQ}(\pi)$ is either $p$ or $p^2$, according to the decomposition  of $p$ in $\mathcal{O}_K$. \\
\noindent $a$) Assume $\sqrt{M(K)}< 1-\frac 1 {\lambda^2}$. Since $K$ is norm Euclidean, for every $ \alpha\in K_{v_0}$ there is a representative $\beta$ of $\frac \alpha\pi\pmod{\calo_{v_0}}$ such that $\beta\in \calo_{K,\{v_0\}}$ and  $|N_{K/\QQ}(\beta)|_\infty \leq M(K)$. Therefore, applying the same construction as in Theorem \ref{teo:euclidean}, there is a type $\tau=(K,\mathfrak{P},s)$ such that $|N_{K/\QQ}(a)|_\infty \leq M(K)|N_{K/\QQ}(\pi)|_{\infty}$ for every $a\in\calY_s$ such that $|a|_{v_0}>1$, that is $|a|_\infty \leq \sqrt{M(K)}\lambda<\lambda-\frac 1\lambda$, hence
$$\theta(a)\leq \theta(\sqrt{M(K)}\lambda)<\lambda, $$ and the claim follows by Theorem \ref{teo:FCcondition} $b$). \\
\noindent$b$) Arguing as above we see that there is a complete set of representatives $\mathcal{R}$  of $\mathcal{O}_K/\mathfrak{P}$ such that $|c|_\infty\leq \sqrt{M(K)}\cdot \lambda$ for every $c\in\mathcal{R}$. Let $L=\max\{|c|_\infty\ |\ c\in\mathcal{R}\}$. Then by hypothesis
$$L< (\lambda - 1)\left (1-\frac 1 {\lambda^2}\right ),$$
and if $\sigma$ is the complex conjugation, then $|x^\sigma|_\infty = |x|_\infty$ for every $x\in K$, therefore Theorem \ref{cor:FEC} $b$) can be applied to the type $(K,\pi,\mathcal{R})$, concluding the proof.
\end{proof}
The following list summarises the behaviour of CFF property for imaginary norm Euclidean fields $K=\QQ(\sqrt{-D})$:
\begin{center}
 \begin{tabular}{ l l | c| c | c |c |c} \\ 
 & $D$ & 1 & 2 & 3 & 7 & 11\\
 \hline
 $\CFF$ property & for $p\geq $ & 3 & 5 & 2 & 3& 7\\
 $\CFF$ special type & for $p\geq $& 7& 23& 11 & 13& 127
 \end{tabular}
\end{center}
\medskip
\subsection{Real norm Euclidean quadratic fields: some explicit constructions}

It is well known that a real quadratic field $\QQ(\sqrt{D})$ is norm Euclidean if and only if $D=2, 3, 11, 13, 17, 19, 21, 29, 33, 37, 41, 57, 73$ (see for example \cite{Hardy_Wright}). In \cite{Eggleton1992}, the authors give an explicit proof of the theorem by showing that, in each of these cases, the fundamental region is covered by (finitely many) unit neighborhoods of the plane, giving the precise list for every of these fields. Using this in combination with the construction of the proof of Theorem \ref{teo:euclidean}, one can show how to  construct explicitly a $\mathfrak{P}$-adic floor function for a prime ideal $\mathfrak{P}$ of $\calo_K$.

Let $K=\QQ(\sqrt{D})$ be a real norm Euclidean field; we consider the plane embedding given by
\begin{align*} j: K& \longrightarrow  \RR^{2}\\\
a+b\sqrt D &\longmapsto (a,b);
\end{align*}
this gives a representation of the elements of $K$ as the points of the plane with rational coordinates.

Under this plane embedding, the algebraic integers correspond to the lattice points $\ZZ^2$, if $D \equiv 2,3 \pmod 4$, and to the mid-lattice points $\frac{1}{2}\ZZ^2$ if $D \equiv 1 \pmod 4$.

For any $\lambda \in \calo_K$, we define the neighborhood of $\lambda$ in $K$ of radius $\epsilon$ to be the set
$$ V_{\epsilon}(\lambda)=\{\beta \in \QQ(\sqrt{D})\ \ |\ \ |N_{K/\QQ}(\beta-\lambda)| < \epsilon\}; $$
using the plane embedding, this maps to 
$$ V_{\epsilon}(x,y)=\{(r,s)\in \QQ^2 \ \ |\ \ |(r-x)^2-D(s-y)^2|<\epsilon\}, $$
where $(x,y)=j(\lambda)$. Notice that these are infinite $X$-shaped regions in the plane bounded by conjugate hyperbolas.

It is then clear that, since we are assuming $K$ norm Euclidean, each $\beta \in \QQ(\sqrt D)$ lies in the neighborhood $V_{\epsilon}(\lambda)$ for some $\lambda \in \calo_K$, i.e. each point $(r,s)\in \QQ^2$ lies in some neighborhood $V_{\epsilon}(x,y)$ in the plane, where $(x,y)=j(\lambda)$ for some $\lambda \in \calo_K$. 
\medskip

Let $\mathfrak{P}$ be a prime ideal in $\calo_K$; we can associate to every generator $\pi \in \mathfrak{P}$ a type $\tau_{\pi}=(\QQ, \mathfrak{P}, s_{\pi})$, where the floor function $s_{\pi}$ is defined by the following algorithm.
Given a coset $\alpha + \mathfrak{P} \calo_{v_0}$ in $K_{v_0}$, we can find, by strong approximation, an element $\alpha'\in K$ belonging to this coset such that $|\alpha'|_{v}<1$ for every non-archimedean $v\in \calM_K \setminus \{v_0\}$; in particular, $\alpha' \in \calo_K[\frac{1}{\pi}]$. We can now translate $\frac{\alpha'}{\pi}$ by a suitable element  $\mu\in \calo_K$ so that $i(\beta):=i(\alpha'-\mu)$ belongs to the region
$$ F(D):=\left \{(r,s)\in \QQ^2 \ \ |\ \ -\frac{1}{2}<r\le \frac{1}{2},\ -\frac{1}{2}<s\le \frac{1}{2}\right \}, $$
and such $\beta$ is unique. We call $F(D)$ fundamental region. Notice that $\alpha'\equiv \pi \beta \pmod {\mathfrak{P}}$. By \cite{Eggleton1992}, we have that $F(D)$ is covered by a finite number of neighborhoods or radius $\epsilon <1$ (depending of $D$) $V_{\epsilon}(x_k,y_k)$; hence, $j(\beta)$ lies in (almost) one of these neighbourhood. We choose a neighbourhood $V_{\epsilon}(x', y')$ such that $j(\beta)$ lies in it and, for every $\gamma \in \alpha + \mathfrak{P}\calo_{v_0}$, we put 
$$ s_{\pi}(\gamma):= \pi(\beta-j^{-1}(x', y')). $$

\begin{example}
Let us consider the case $D=17$; since $D\equiv 1 \pmod 4$, then $\calo_K=\ZZ\left [ \frac{1+\sqrt D}{2} \right ]$. 
Let us divide the fundamental region $F(17)$ into six subsets, namely:
\begin{itemize}
    \item $F_1=\{(x,y)\in \QQ^2 \ | \ 0<r\le 1/2,\ -1/4<s\le 1/4\};$
    \item $F_2=\{(x,y)\in \QQ^2 \ |\ -1/2<r\le 0,\ -1/4<s\le 1/4\};$
    \item $F_3=\{(x,y)\in \QQ^2 \ |\ 0<r\le 1/2,\ 1/4<s\le 1/2\};$
    \item $F_4=\{(x,y)\in \QQ^2 \ |\ 0<r\le 1/2,\ -1/2<s\le -1/4\};$
    \item $F_5=\{(x,y)\in \QQ^2 \ |\ -1/2<r\le 0,\ 1/4<s\le 1/2\};$
    \item $F_6=\{(x,y)\in \QQ^2 \ |\ -1/2<r\le 0,\ -1/2<s\le -1/4\}$.
\end{itemize}
Then, $F(17)$ is equal to the union of these regions, and the union is disjoint, hence every $\beta \in F(17)$ belongs to one $F_k$. We have now to associate to every $F_k$ a unit neighborhood $V(x,y)$ that covers the corresponding region; this can be of course done in many ways. 

We use an argument analogous to \cite{Eggleton1992}. 
By easy calculations, we have that the point $(1/2, 1/4)\in F_1$ lies of the top boundary of the neighborhood $V_{13/16}(1,0)$, hence the preimage of every point in $F_1$ satisfies $N_{K/\QQ}(\beta-1)\le 13/16$. Similarly, the point $(1/2,1/4)$ lies on the bottom boundary of the neighborhood $V_{13/16}(1,1/2)$, hence $F_3$ is contained in its closure. Using the symmetry properties of $F(17)$, it is easy to see that $F_2 \subset \overline{V_{13/16}(-1,0)}$, $F_4\subset \overline{V_{13/16}(1,-1/2)}$, $F_5\subset \overline{V_{13/16}(-1,1/2)}$ and $F_6\subset \overline{V_{13/16}(-1,-1/2)}$. 
For every $k=1, \ldots, 6$, let us denote by $\delta_k$ the preimage in $\calo_K$ of the center of the corresponding neighborhood, i.e. $\delta_k:=j^{-1}(x_k,y_k)$.
Using this, we can perform the algorithm described above. 

Given a prime ideal $\mathfrak{P} \subset \calo_{v_0}$, choose a suitable generator $\pi$ of $\mathfrak{P}\calo_{v_0}$. Then, for every coset $\alpha+\mathfrak{P}\calo_{v_0}$, choose $\alpha' \in \alpha+\mathfrak{P}\calo_{v_0} \cap \calo_K[1/\pi]$, and translate it by an element $\mu\in \calo_K$ so that the image of $\beta:=\alpha-\mu$ lies in the fundamental region; then $j(\beta)\in F_k$ for some $k=1, \ldots, 6$. 

Take any $\gamma\in \alpha+\mathfrak{P}\calo_{v_0}$; then, we denote by 
$$ s_{\pi}(\gamma):=\pi (\beta-\delta_i). $$ 

Let us show for example that, if $p$ is an odd prime which is inert in $\calo_K$, then this choice of the floor function gives rise to a type satisfying $\CFF$ property.

If $p$ is inert, then we can take $\pi=p$, hence $N_{K/\QQ}(\mathfrak{P})=p^2$. To apply Theorem \ref{teo:FCcondition}, we have to estimate $\theta(a^{\sigma})$ for every $a\in \calY_s$ such that $|a|_{v_0}>1$ and every embedding $\sigma$ of $K$ into $\RR$, which are exactly the identity and the one sending $\sqrt{17}$ to $-\sqrt{17}$. By the above choice of the neighborhoods covering the fundamental region and the corresponding construction of the floor function, we have that, for every $\beta\in F_k$ and for every centre of the corresponding neighborhood $\delta_k$, $|\beta^{\sigma}-\delta_k^{\sigma}|\le \sqrt{5}/4$, and $N_{K/\QQ}(\beta^{\sigma}-\delta_k^{\sigma})\le 1/4$, hence for every $a$, we have 
$N_{K/\QQ}(a)\le \frac{13}{16}p^2$.

It follows that
$$ \prod_{\sigma \in \Sigma}\theta(a^{\sigma}) < \prod_{\sigma\in \Sigma}(1+|a^{\sigma}|)\le 1+\frac{\sqrt 5}{2}p + \frac{13}{16}p^2.$$
Since $|a|_p^{d_p}\ge p^2$, we have that $\nu_{\tau_\pi}<1$ if 
$$ 1+\frac{\sqrt 5}{2}p + \frac{13}{16} p^2< p^2,$$
which holds for every prime $p\ge 3$.
We finally point out that a similar argument involving another choice of the generator of the prime ideal $\mathfrak P$ can be used in the case where $p$ split, as done for example in Lemma \ref{lem:199}.
\end{example}

\subsection{The $\CFF$ property for $\QQ \left (\sqrt{2} \right )$}

It is well known that $K=\QQ(\sqrt{2})$  is norm Euclidean. We can regard $K$ as a subfield of $\RR$, so that $\Sigma=\{id,\sigma\}$, and $\sigma$ is the embedding sending $\sqrt 2$ in $-\sqrt 2$. The fundamental unit is $u=1+\sqrt{2}$. \\
This section is devoted to prove the following result.
\begin{theorem}\label{teo:Qsqrt2}
The field $\QQ(\sqrt{2})$ has the $\mathfrak{P}$-adic $\CFF$ property, for every prime ideal $\mathfrak{P}$ of odd residual characteristics.
\end{theorem}
Let $\mathfrak{P}$ be a prime ideal in $\calo_K$ with  residual characteristics $p>2$.  We can  associate to every generator $\pi$ of $\mathfrak{P}$ a (not uniquely determined) type $\tau_\pi=(\QQ,\mathfrak{P},s_\pi)$  as follows: 
choose a coset $\alpha+\mathfrak{P}\calo_{v_0}$  in $ K_{v_0}$; choose $b\in K$  such that $|b-\alpha|_{v_0}< 1$ and $|b|_v\leq 1 $ for every non archimedean $v\in\calM_K\setminus\{v_0\}$; then we can write $b=\frac A {\pi^k}$ for some $k\geq 1$. By dividing $A$ by $\pi^{k+1}$, we can find $\beta\in\calo_K$ such that 
$\frac A {\pi^{k+1}}=\beta+\gamma$
with $\gamma=x+y\sqrt{2}$, $x,y\in\QQ$, $|x|_\infty, |y|_\infty\leq \frac 1 2 $. Then, we see that $\gamma\in \calo_K[\frac 1 \pi]$ and, if we put $a=\pi\gamma$, we find $|a-\alpha|_{v_0}< 1$ and $|a|_v\leq 1 $ for every non archimedean $v\in\calM_K$. We define $s_\pi(\alpha+\mathfrak{P})=a$.\\

We denote by $N_{K/\QQ}(\mathfrak{P})=p^f$ ($f\in\{1,2\}$), and we put $\lambda=|\pi|_\infty$, $\lambda_\sigma=|\pi^\sigma|_\infty=\frac {p^{f}} \lambda $. Then we have
$$|a|_\infty \leq \frac 1 2\left (1+\sqrt{2} \right )\lambda,\quad |a^\sigma |_\infty \leq \frac 1 2 \left (1+\sqrt{2} \right )\lambda_\sigma;$$ and
$$N_{K/\QQ}(a)=p^f(x^2-2y^2)\leq \frac 1 2 p^f.$$
Hence,
\begin{align*}
    \theta(a)\theta(a^\sigma)<(|a|_\infty+1)(|a^\sigma|_\infty+1) &\leq |N_{K/\QQ}(a)|_\infty +\ |a|_\infty + |a^\sigma |_\infty+1\\
    &\leq \frac 1 2 p^f+ \frac 1 2 \left (1+\sqrt{2} \right ) \left (\lambda +\frac {p^f} \lambda \right )+1.
\end{align*}
By imposing the last quantity to be less than $p^f$, we obtain
\begin{equation*}\label{eq:disug} 
\lambda+\frac {p^f}\lambda< (\sqrt{2}-1)(p^f -2),
\end{equation*}
that is $F_p(\lambda)< 0$, where
$$F_p(X)=X^2-(\sqrt{2}-1)(p^f -2)X+p^f.$$
It follows by Theorem \ref{teo:FCcondition} that the type $\tau_\pi$ has the $\CFF$ property for every $\lambda$ satisfying $F_p(\lambda)<0$.    

\begin{lemma}\label{lem:199} Assume that the residual characteristics $p$ satisfies:
\begin{itemize}
    \item $p\geq 41$ if $p$ splits  in $\calo_K$; 
    \item $p\geq 11$ if $p$ is inert  in $\calo_K$; 
\end{itemize} then, the $\mathfrak{P}$-adic $\CFF$ property holds for $K$.
\end{lemma}
\begin{proof}
If $p$ is inert in $\calo_K$, take $\pi=p$ in the above considerations. Then $\lambda=p$ and we see that $F_p(p)<0$ except for $p=3,5$. Therefore the type $\tau_p$ satisfies the $\CFF$ property.\\ Assume now that $p$ splits and let $\pi$ be the unique generator of $\mathfrak{P}$ such that $0<\pi\leq  \sqrt{p}$ and $u\pi>\sqrt{p}$. It follows that $\frac {\sqrt{p}} u <\lambda <\sqrt{p}$ and is straightforward to verify that $F_p(\lambda)< 0$ for $p\geq 71$. Therefore, the type $\tau_\pi$ satisfies the $\CFF$ property for a splitting $p\geq 71$.\\
It remains to consider the cases $p\in\{31,41,47\}$, $p$ splitting. In this cases it is straightforward to see that by setting $\pi=1+4\sqrt{2}, 3+5\sqrt{2}, 5+6\sqrt{2}$ respectively we fulfill the requirement $F_p(\lambda)<0$, so that we get types $\tau_\pi$ satisfies the $\CFF$ property also in these last cases.
\end{proof}

In order to complete the proof of Theorem \ref{teo:Qsqrt2} it remains to consider the residual characteristics 
$p$ in the set $S=\{3,5,7,17,23\}$. Notice that
\begin{itemize}
    \item $p$ is inert for $p=3,5$;
    \item $p$  splits for $p=7,17,23$.
\end{itemize}
Moreover, in these cases $F_p(X)$ is strictly positive over $\RR$, so that the above technique is not applicable. For the prime in $S$ the following proposition holds.
\begin{proposition}
Let $p\in S$  and $\mathfrak{P}$ be a prime ideal in $\calo_K$ with residual characteristics $p$. There exists a generator $\pi$ of $\mathfrak{P}$ such that the type $\tau_\pi=(\QQ(\sqrt{2}),\mathfrak{P},s_\pi)$ satisfies the $\CFF$ property. 
\end{proposition}
\begin{proof}
Let us choose a generator $\pi$ of $\mathfrak{P}$ minimizing the distance from $\sqrt{p}$; for example,
\begin{itemize}
    \item for $p=3,5$ we set $\pi=p$; 
    \item for $p=7,17,23$, we choose $\pi$ such that $1<\pi<p$ and $\lambda+ \lambda_\sigma$ is minimum, where $\lambda=|\pi|_\infty$.
\end{itemize} 
Let $\alpha\in K_{v_0}$ be such that $|\alpha|_{v_0}>1$ and put $a=s_\pi(\alpha)$. Our goal is to show that there exists $\epsilon<1$ such that
$$|\theta(a)\theta(a^\sigma)|\leq \epsilon p^f,$$
so that we can apply Theorem \ref{teo:FCcondition} to prove the theorem. By recalling the construction of $\tau_\pi$,
 this amounts to show that the inequality 
$$\tilde{Z}(x,y)=\theta((x+\sqrt{2}y)\pi)\theta((x-\sqrt{2}y)\pi^\sigma)\leq \epsilon p^f$$
holds on the square $\{(x,y)\ |\ -\frac 1 2\leq  x,y\leq \frac 1 2\}$. By symmetry, it suffices to bound $\tilde{Z}$ on the square $\{(x,y)\ |\ 0 \leq  x,y\leq \frac 1 2\}$. By denoting $X=x+y\sqrt{2}$, $Y=x-y\sqrt{2}$, this amounts to show that 
$$Z(X,Y)=\theta(X\pi)\theta(Y\pi^\sigma)\leq \epsilon p^f, $$
on the parallelogram $\mathcal P$ delimited by the lines
$$r_1:X-Y=0, \quad\quad r_2:X-Y=\sqrt{2},\quad\quad s_1:X+Y=0, \quad\quad s_2:X+Y=1.$$
Since $Z$ is increasing in the variable $X$, its maximum on $\mathcal P$ must be achieved on the sides lying on the two lines $r_2$ and $s_2$. We are thus led to show that the two functions 
\begin{align*}
    Z_1(X)& =\theta(X\pi)\theta((\sqrt{2}-X)\pi^\sigma)\\
    Z_2(X)& =\theta(X\pi)\theta((X-1)\pi^\sigma)
\end{align*}
are bounded by $\epsilon\, p$ on the two intervals $I_1=[\frac {\sqrt{2}} 2,\frac {1+\sqrt{2}} 2 ]  $ and $I_2= [\frac 1 2,\frac {1+\sqrt{2}} 2 ]  $ respectively. The function $Z_1$ is increasing on $I_1$, therefore it achieves its maximum in $X=\frac {1+\sqrt{2}} 2$. 

We split the interval $I_2$ in two sub-intervals $I_2'=[\frac 1 2,1]$ and $I_2''=[1, \frac {1+\sqrt{2}} 2]$; the function $Z_2$ is increasing on $I_2''$, therefore its maximum is achieved in $X=\frac {1+\sqrt{2}} 2$. It follows that the values of $Z$ over $\mathcal P$ are bounded by the maximum between $Z_1(\frac {1+\sqrt{2}} 2)$, $Z_2\left (\frac 1 2 \right )$, $Z_2(\frac {1+\sqrt{2}} 2)$ and the local extremal values of $Z_2$. 
A direct calculation shows that the derivative of $Z_2$ has a unique zero in 
$$T=\frac 1 2 \left( 1+ 4\frac{\pi^2-(\pi^\sigma)^2} {p^2}\right), $$ 
and it is straightforward to show that 
$$ Z_1\left (\frac {1+\sqrt{2}} 2\right ), Z_2\left (\frac 1 2 \right ), Z_2\left(\frac {1+\sqrt{2}} 2\right ), Z_2(T)\leq \epsilon p^f,$$
for a suitably chosen $\epsilon <1$. The claim follows by applying Theorem \ref{teo:FCcondition}.
\end{proof}

\section{$\CFF$ property, class group and Euclidean ideal classes}

All the previous results about the  $\CFF$ property concern norm Euclidean number fields. In the following we investigate more deeply the relationship between $\CFF$ property and the structure of the ideal class group $\Cl(K)$, with the aim to show that our hypothesis is not too restrictive. In what follows, for every fractional ideal $I$, we will benote its class in $\Cl(K)$ by $[I]$.
\begin{proposition}\label{prop:Clciclico} Assume that the field $K$ satisfies the $\mathfrak{P}$-adic $\CFF$ property. Then, $\Cl(K)$ is cyclic generated by $[\mathfrak{P}]$. In particular, if $\mathfrak{P}$ is principal then $\calo_K$ is a $\mathrm{PID}$.
\end{proposition}

\begin{proof}
Let $n_0$ be the order of  $[\mathfrak{P}]$ in $\Cl(K)$ and let $\eta$ be a generator of $\mathfrak{P}^{n_0}$; then, $\calo_{K,\{v_0\}}=\calo_K[\frac 1 \eta]$.
Since $K$ satisfies the $\mathfrak{P}$-adic $\CFF$ property, then every element of $K$ can be expressed as a quotient $\frac{A}{B}$ of elements in $\calo_{K,\{v_0\}}$ such that $A$ and $B$ are coprime in $\calo_{K,\{v_0\}}$, i.e., the ideal generated by $A$ and $B$ in $\calo_{K,\{v_0\}}$ is trivial. This implies that a power of $\eta$ can be written as an $\calo_K$-linear combination of $A$ and $B$, hence the class of the fractional ideal  generated by $A$ and $B$ is a power of $[\mathfrak{P}]$. Now, let $I$ be any ideal of $\calo_K$; then $I$ admits a set of generators of cardinality two, i.e. $I=(\alpha,\beta)$, see e. g. \cite[\S 1.1 Cor. 5]{Narkiewicz2004}). By applying the above argument to $\frac \alpha\beta$ we find by \cite[Proposition 18]{Cooke1976b} that $[I]$ is a power of $[\mathfrak{P}]$ in $\Cl(K)$.
\end{proof}

\begin{corollary} Assume that  the field $K$ satisfies the $\mathfrak{P}$-adic $\CFF$ property for all but finitely many prime ideals $\mathfrak P$. Then $\calo_K$ is a $\mathrm{PID}$.
\end{corollary}
\begin{proof}
This is a direct consequence of Proposition \ref{prop:Clciclico} since every class of ideals contains infinitely many prime ideals (see for example \cite[\S 7.2, Corollary 6 and p.93 for notation]{Narkiewicz2004}). In particular, our hypothesis ensures that $K$ satisfies the $\mathfrak{P}$-adic $\CFF$ property for at least one principal ideal $\mathfrak{P}$. \end{proof}

 Proposition \ref{prop:Clciclico} allows to give examples of number fields $K$ for which the $\CFF$ property fails to hold. For example, if $K=\QQ(\sqrt{-D})$ is an imaginary quadratic field and $r$ is the number of the distinct prime divisors of its discriminant, then $\Cl(K)$ maps onto $(\ZZ/2\ZZ)^{r-1}$ (see \cite[Theorem 8.23]{Narkiewicz2004}); hence it is not cyclic for $r\geq 3$. It follows that such fields do not satisfy the $\mathfrak{P}$-adic $\CFF$ property for any prime $\mathfrak{P}$. \begin{remark}
 If $\Cl(K)$ is cyclic generated by $[\mathfrak{P}]$ then $\calo_{K,\{v_0\}}$ is a $\mathrm{PID}$. By assuming the Generalized Riemann Hypothesis, it is Euclidean \cite[Theorem 10]{Treatman1998} (not necessarily with respect to the norm), and therefore it satisfies the Euclidean chain condition \cite[\S 14.1]{OMeara1965}, i.e. every $\alpha\in K$ can be expressed as a finite continued fraction
 \begin{eqnarray}\label{eq:ECC}
 \alpha&= &
b_0+\cfrac{1}{b_1+\cfrac{1}{\ddots+\cfrac{1}{b_n}}}=[b_0,\ldots, b_n], \hbox{ with $b_i\in \calo_{K,\{v_0\}}$ for $i\geq 0$. }
\end{eqnarray}
 If the $\CFF$ property holds for $K$, then every $\alpha\in K$ admits a representation of the form $\eqref{eq:ECC}$ satisfying the following additional properties, i.e.
 \begin{itemize}
     \item $|b_i|_{v_0}>1$ for $i>1$;
     \item if $|b_0|_{v_0}<1$, then $b_0=0$;
     \item if $b_i\equiv b_j\pmod{\mathfrak{P}}$, then $b_i=b_j$.
 \end{itemize}

 \end{remark}

 We conclude our work by presenting two results which attempt to go beyond the Euclidean assumption that we made along this paper. The next Theorem \ref{teo:euclideanclass} extends Theorem \ref{teo:euclidean} to prime ideals lying in a \emph{norm Euclidean ideal class} in the sense of \cite{LenstraJr1979}. For a fractional ideal $I$ of $K$ we will consider the following property:
 $$ \textrm{\textit{if $\beta\in K$, there exists $\alpha\in I$ such that $|N_{K/\QQ}(\alpha-\beta)|<N_{K/\QQ}(I)$}}.$$
 
 This property depends only on the class $[I]$ in $\Cl(K)$; such a class is called a \emph{norm Euclidean class}. Theorem 0.3 of \cite{LenstraJr1979} shows that $\Cl(K)$ contains at most one norm Euclidean class and, if there is one, it generates $\Cl(K)$. As it is shown in \cite{McGown2012}, for an ideal class $\calC$ one can give an analogous definition of Euclidean minimum $M_\calC$ and of inhomogeneous Euclidean minimum $\overline{M} _\calC$; moreover, if the rank $r=r_1+r_2-1$ of units is $>1$, one has $M_\calC=\overline{M}_\calC$. In particular if $r>1$, then a class $\calC$ is norm Euclidean if and only if $\overline{M}_\calC<1$.
 
 We notice that non principal Euclidean classes exist for example for fields like $\QQ(\sqrt{-15})$ and $\QQ(\sqrt{-20})$ (see \cite[Prop. 2.1]{LenstraJr1979}), and $\QQ(\sqrt{10})$, $\QQ(\sqrt{15})$, $\QQ(\sqrt{85})$ (see \cite[2.5]{LenstraJr1979}); other examples can be found in \cite{Lezowski2012}.
 
 For number fields having norm Euclidean class with inhomogeneous Euclidean minimum $<1$, one can generalise the proof of Theorem \ref{teo:euclideanclass} to prove the following result.
 
 \begin{theorem}\label{teo:euclideanclass} Let $K$ be a number field and assume that $K$ has a norm Euclidean ideal class $\calC$ such that $\overline{M}_\calC <1  $. Then $K$ satisfies the $\mathfrak{P}$-adic $\CFF$ property for all but finitely many $\mathfrak{P}\in\calC$.
 \end{theorem}
 
 \begin{remark} If $r=r_1+r_2-1 >1$, then it is known that $M_\calC = \overline{M}_\calC <1  $  for every norm Euclidean class $\calC$. On the other hand, there are examples where $M_\calC =1  $, when $r\leq  1$ (see \cite{Lezowski2012, McGown2012}). It would be nice to replace the hypothesis $\overline{M}_\calC <1  $ in Theorem \ref{teo:euclideanclass} with the more natural hypothesis ${M}_\calC <1  $; however, unlike the case $\calC=\calo_K$, we do not know if an analogue of Theorem \ref{teo:cerri1} is true for an arbitrary norm Euclidean ideal classes. 
 \end{remark}
 
 \begin{proof}[Proof of Theorem \ref{teo:euclideanclass}]
  Fix $\mathfrak{Q}$ a prime ideal in $\calC$ and let  $\calD=\calD_{\mathfrak{Q}}\subseteq \RR^{r_1}\times\CC^{r_2}$ be a fundamental domain for the lattice $i(\mathfrak{Q})$. Since $\calD$ is compact, there exists  a finite set $\alpha_1,\ldots,\alpha_s\in\mathfrak{Q}$ and $\epsilon<1$ such that $\calD\subseteq \bigcup_{i=1}^s U_{ \epsilon N_{K/\QQ}(\mathfrak{Q})}(\alpha_i)$. 
 
 For every prime ideal $\mathfrak{P}$ in $\calC$ we choose, by Lemma \ref{lem:univin}, an element $\gamma_{\mathfrak{P}}\in K$ such that $\gamma_{\mathfrak{P}}\mathfrak{Q}=\mathfrak{P}$ and $|\gamma_{\mathfrak{P}}^\sigma|_\infty \leq T_{0} \sqrt[d]{|N_{K/\QQ}(\gamma)|_{\infty}}$ for every $\sigma\in\Sigma$.
 Then, $\calD_\mathfrak{Q}=i(\gamma_{\mathfrak{P}}) \calD$ is a fundamental domain for $i(\mathfrak{P})$. We construct a type $\tau_{\mathfrak{P}}=(K,\mathfrak{P},s_{\mathfrak{P}})$ by mimicking the proof of Theorem \ref{teo:euclidean}: given a coset $\alpha+\mathfrak{P}_{v_{\mathfrak{P}}}$ in $K_{v_{\mathfrak{P}}}$ we find a representative $\beta\in \calD_{\mathfrak{P}}\cap \calo_{K,\{v_\mathfrak{P}\}}=\gamma_{\mathfrak{P}}(\calD\cap \calo_{K,\{v_\mathfrak{P}\}})$ ; then  $|N_{K/\QQ}(\gamma_{\mathfrak{P}}^{-1}\beta-\alpha_i)|_{\infty}< \epsilon N_{K/\QQ}(\mathfrak{P}) $ for some $i=1,\ldots,s$, so that $|N_{K/\QQ}(\beta-\gamma_{\mathfrak{P}}\alpha_i)|_{\infty}< \epsilon N_{K/\QQ}(\mathfrak{P}) $. We define $$a=s_\mathfrak{P}(\alpha+\mathfrak{P}_{v_{\mathfrak{P}}})=\beta-\gamma_{\mathfrak{P}}\alpha_i. $$
 By construction we have
 \begin{equation}\label{eq:limitanorma} 
 |N_{K/\QQ}(a)|_{\infty}< \epsilon N_{K/\QQ}(\mathfrak{P}).
 \end{equation}
 Moreover, since $\calD$ is bounded and the $\alpha_i$ are finitely many, there exists a constant $C$ (depending on $\mathfrak{Q}$) such that 
 \begin{equation}\label{eq:limitacomponenti} 
 |a^\sigma |_\infty \leq C|\gamma^\sigma|_\infty\leq CT_0\sqrt[d]{|N_{K/\QQ}(\gamma)|_{\infty}}=\frac {CT_0}{\sqrt[d]{N_{K/\QQ}(\mathfrak{Q})}}\sqrt[d]{N_{K/\QQ}(\mathfrak{P})}.\end{equation} 
 As in the proof of Theorem \ref{teo:euclidean}, from \eqref{eq:limitanorma} and \eqref{eq:limitacomponenti} we conclude that 
 $$
 \prod_{\sigma\in\Sigma}\theta(a^\sigma)< \prod_{\sigma\in\Sigma}(|a^\sigma|+1) <  \epsilon N_{K/\QQ}(\mathfrak{P})\quad\hbox{for } N_{K/\QQ}(\mathfrak{P})\gg 0.$$
 Therefore, $\nu_{\tau_\mathfrak{P}}<1$ for $N_{K/\QQ}(\mathfrak{P})\gg 0$, so that $\tau_\mathfrak{P}$ satisfies the $\CFF$ property by Theorem \ref{teo:FCcondition}.
 \end{proof}
 
\noindent For an arbitrary number field $K$ we have the following result:
 
 \begin{theorem}\label{teo:generalnumberfields} Let $K$ be a number field. Then, for all but finitely many prime ideals $\mathfrak{P}$, there exists a type $\tau_\mathfrak{P}=(K,\mathfrak{P},s_\mathfrak{P})$ with the following property:\\
 if $\alpha\in K$ and the continued fraction $\tau_\mathfrak{P}$-expansion $[a_0,a_1,a_2,\ldots ]$ for $\alpha$ is infinite, then $a_j\in \mathfrak{P}^{-1}$ for infinitely many $j\in \NN$.
\end{theorem}
 
\begin{proof}
Fix an ideal class $\calC$ and let $\overline{M}_\calC$ be the inhomogeneous euclidean minimum of $\calC$; let $\epsilon > \overline{M}_\calC$. As in the proof of Theorem \ref{teo:euclideanclass}, we can fix a prime ideal $\mathfrak{Q}\in\calC $ and construct for every $\mathfrak{P}\in\calC$ a type $\tau_\mathfrak{P}=(K,\mathfrak{P},s_\mathfrak{P})$ such that every $a$ belonging to the image of $s_\mathfrak{P}$ satisfies conditions \eqref{eq:limitanorma} and \eqref{eq:limitacomponenti}.
Then, it is easy to see that there exists a suitable constant $T_\calC<1$, depending only on the class $\calC$ such that, for every $\delta >0$,
$$\prod_{\sigma\in\Sigma}\theta(a^\sigma)\leq  \prod_{\sigma\in\Sigma}(|a^\sigma|+1) <  \epsilon N_{K/\QQ}(\mathfrak{P})< T_\calC N_{K/\QQ}(\mathfrak{P})^{1+\delta} \quad\hbox{for } N_{K/\QQ}(\mathfrak{P})\gg 0.$$
Since $\Cl(K)$ is finite, we deduce that there is a constant $T<1$ depending only on $K$ such that, for all but finitely many prime ideal $\mathfrak{P}$ of $\calo_K$ and every $a$ in the image of $s_\mathfrak{P}$, we have
$$\prod_{\sigma\in\Sigma}\theta(a^\sigma)< T N_{K/\QQ}(\mathfrak{P})^{1+\epsilon}.$$
Now assume by contradiction that there exists $\alpha\in K$ such that the continued fraction $\tau_\mathfrak{P}$-expansion $[a_0,a_1,a_2,\ldots ]$ of $\alpha$ is infinite and such $a_j\not\in \mathfrak{P}^{-1}$ for all but finitely many $j\in \NN$. Possibly taking the tail of the expansion, we can assume that $a_j\not\in \mathfrak{P}^{-1}$ for every $j\in \NN$; this implies that
$$ \sup_j\left\{ \frac{\prod_{\sigma\in\Sigma}\theta(a_j^\sigma)} 
{|a_j|_{v_\mathfrak{P}}}\right\} \leq \sup_j\left\{ \frac{\prod_{\sigma\in\Sigma}\theta(a_j^\sigma)} 
{N_{K/\QQ}(\mathfrak{P})^2}\right\} \leq T<1,$$
which gives a contradiction by Theorem \ref{teo:FCcondition}. 
\end{proof}

 \section*{Acknowledgements} We thank Marzio Mula and Francesco Veneziano for useful conversations. 
\bibliographystyle{abbrv}
\bibliography{MCF}

\end{document}